\documentclass{article}
\usepackage[latin1]{inputenc}
\usepackage[T1]{fontenc}
\usepackage{amsmath,amssymb}
\usepackage{amsfonts,amsthm}

\usepackage{geometry}
\usepackage{color}
   \definecolor{cites}{rgb}{0.75 , 0.00 , 0.00}  
   \definecolor{urls} {rgb}{0.00 , 0.00 , 1.00}  
   \definecolor{links}{rgb}{0.00 , 0.00 , 0.5}   
  \definecolor{gray}{rgb}{0.5,0.5,.5}
\usepackage{graphicx}
\usepackage{cancel}
\usepackage[plainpages=false,colorlinks=true, citecolor=cites, urlcolor=urls, linkcolor=links, breaklinks=true]{hyperref}

\allowdisplaybreaks

\newcommand{\B}{\mathbb{B}}
\newcommand{\C}{\mathbb{C}}
\newcommand{\D}{\mathbb{D}}
\newcommand{\N}{\mathbb{N}}
\newcommand{\R}{\mathbb{R}}

\newcommand{\Bc}{\mathcal{B}}
\newcommand{\Kc}{\mathcal{K}}
\newcommand{\Lc}{\mathcal{L}}

\newcommand{\Tf}{\mathfrak{T}}

\newcommand{\BDO}{\textup{BDO}}

\renewcommand{\epsilon}{\varepsilon}

\newcommand{\vertiii}[1]{{\left\vert\kern-0.25ex\left\vert\kern-0.25ex\left\vert #1
    \right\vert\kern-0.25ex\right\vert\kern-0.25ex\right\vert}}
\newcommand{\vertiiis}[1]{{\vert\kern-0.25ex\vert\kern-0.25ex\vert #1
    \vert\kern-0.25ex\vert\kern-0.25ex\vert}}

\DeclareMathOperator{\Aut}{Aut}

\DeclareMathOperator{\BUC}{BUC}

\DeclareMathOperator{\diam}{diam}
\DeclareMathOperator{\dist}{dist}
\DeclareMathOperator{\ess}{ess}
\DeclareMathOperator{\MO}{MO}
\DeclareMathOperator{\Osc}{Osc}
\DeclareMathOperator*{\slim}{s-\lim}
\DeclareMathOperator{\spec}{sp}
\DeclareMathOperator{\supp}{supp}
\DeclareMathOperator{\VMO}{VMO}
\DeclareMathOperator{\VO}{VO}
\DeclareMathOperator*{\wlim}{w-\lim}

\newcommand{\from}{\colon}

\providecommand{\abs}[1]{\left\lvert#1\right\rvert}
\providecommand{\norm}[1]{\left\lVert#1\right\rVert}
\providecommand{\set}[1]{\left\{ #1\right\}}

\newtheorem{thm}{Theorem}
\newtheorem{mthm}{Theorem}
\newtheorem{lem}[thm]{Lemma}
\newtheorem{prop}[thm]{Proposition}
\newtheorem{cor}[thm]{Corollary}
\newtheorem*{cor*}{Corollary}

\theoremstyle{definition}
\newtheorem{defn}[thm]{Definition}

\theoremstyle{remark}
\newtheorem{rem}[thm]{Remark}

\numberwithin{equation}{section}

\allowdisplaybreaks
\begin{document}
\title{\bf Limit Operators, Compactness and Essential Spectra on Bounded Symmetric Domains}
\author{Raffael Hagger\footnote{Institut f\"ur Analysis, Leibniz Universit\"at, 30167 Hannover, Germany, raffael.hagger@math.uni-hannover.de}}
\maketitle
\vspace{-0.4cm}
\begin{abstract}
This paper is a follow-up to a recent article about the essential spectrum of Toeplitz operators acting on the Bergman space over the unit ball. As mentioned in the said article, some of the arguments can be carried over to the case of bounded symmetric domains and some cannot. The aim of this paper is to close the gaps to obtain comparable results for general bounded symmetric domains. In particular, we show that a Toeplitz operator on the Bergman space $A^p_{\nu}$ is Fredholm if and only if all of its limit operators are invertible. Even more generally, we show that this is in fact true for all band-dominated operators, an algebra that contains the Toeplitz algebra. Moreover, we characterize compactness and explain how the Berezin transform comes into play. In particular, we show that a bounded linear operator is compact if and only if it is band-dominated and its Berezin transform vanishes at the boundary. For $p = 2$ ``band-dominated'' can be replaced by ``contained in the Toeplitz algebra''.


\end{abstract}

\section{Introduction} \label{introduction}

In the introduction of \cite{Hagger} it was mentioned that ``similar results are expected to hold for more general domains'' and that ``there are some open problems in the most general case''. In short, the aim of this paper is to solve these open problems and thus prove the ``similar results''. As it turns out, the solution not only generalizes the domain, but also the set of eligible operators.

Before we jump into details, let us first recall the basic setting. Let $\Omega$ denote a bounded symmetric domain in its Harish-Chandra realization and let $L^p_{\nu} := L^p(\Omega,v_{\nu})$ denote the corresponding $L^p$-space for some weighted Lebesgue measure $v_{\nu}$ and $p \in (1,\infty)$. Now consider the (closed) subspace of holomorphic functions $A^p_{\nu} \subset L^p_{\nu}$ and assume that there is a bounded projection $P_{\nu}$ onto $A^p_{\nu}$. Then for every bounded function $f \from \Omega \to \C$ we may consider the corresponding Toeplitz operator, which is defined by
\[T_fg = P_{\nu}(f \cdot g)\]
for $g \in A^p_{\nu}$. Denote by $\Tf_{p,\nu}$ the Banach algebra generated by all such Toeplitz operators.

A natural (and non-trivial) question to ask is under which conditions a Toeplitz operator $T_f$ is compact (e.g.~\cite{AxZhe,Englis,MiSuWi,MiWi,StroeZhe,Suarez,Zhu88}). For $p = 2$ a satisfactory answer was given by Engli\v{s} in \cite{Englis,Englis_err}, namely, $T_f$ is compact if and only if the Berezin transform of $f$ vanishes at the boundary. In fact, Engli\v{s} showed a little bit more. He showed that if $A$ can be written as a finite sum of finite products of Toeplitz operators, then $A$ is compact if and only if the (generalized) Berezin transform $\Bc(A)$ vanishes at the boundary. This result gives rise to the question whether this is true for all bounded linear operators on $A^p_{\nu}$. One direction is actually quite simple: If $A$ is a compact operator on $A^p_{\nu}$, then $\Bc(A)$ vanishes at the boundary. However, the other direction turns out to be wrong (see e.g.~\cite{AxZhe}). This suggests that there is some condition missing here. For $p = 2$ we observe that the ideal of compact operators has to be fully contained in $\Tf_{2,\nu}$ because $\Tf_{2,\nu}$ is an irreducible $C^*$-algebra and contains non-trivial and hence all compact operators. In the case of the unit ball $\B^n$, Su\'{a}rez (\cite{Suarez}, see \cite{MiSuWi} for the weighted case) proved that this remains true for arbitrary $p \in (1,\infty)$. Hence the new conjecture would read ``$A$ is compact if and only if $A \in \Tf_{p,\nu}$ and $\Bc(A)$ vanishes at the boundary''. For the unit ball this was shown in \cite{MiSuWi,Suarez} and it is widely conjectured that this holds for arbitrary bounded symmetric domains (e.g.~in \cite{MiWi}). We now show this conjecture in the case $p = 2$ and present an alternative description for general $p$ by using band-dominated operators, which were introduced by the author in \cite{Hagger}. More precisely, we show that a bounded linear operator on $A^p_{\nu}$ is compact if and only if it is band-dominated and its Berezin transform vanishes at the boundary:

\renewcommand*{\themthm}{A}
\begin{mthm} \label{A}
An operator $K \in \Lc(A^p_{\nu})$ is compact if and only if $K$ is band-dominated and
\[\lim\limits_{z \to \partial\Omega} (\Bc(K))(z) = 0.\]
\end{mthm}

Using the argument above, ``band-dominated'' can be replaced by ``contained in the Toeplitz algebra'' for $p = 2$:

\begin{cor*}
An operator $K \in \Lc(A^2_{\nu})$ is compact if and only if $K \in \Tf_{2,\nu}$ and $\lim\limits_{z \to \partial\Omega} (\Bc(K))(z) = 0$.
\end{cor*}

In a similar vein Fredholmness of a band-dominated operator can be characterized. Using the techniques developed in \cite{Hagger}, we show that a band-dominated operator is Fredholm if and only if all of its limit operators are invertible, where limit operators occur as strong limits of certain operator nets (see Section \ref{limit_operators} for a precise definition). One of the key parts here is to actually show the existence of these strong limits. For Toeplitz operators on the unit ball this was done in \cite{MiSuWi,Suarez}. However, the proof there involves some direct computations, which are not accessible in the case of general bounded symmetric domains. We thus use the theory of band-dominated operators once again to show the existence of these limit operators. As a bonus, we obtain existence for all band-dominated operators rather than just the Toeplitz algebra. After the existence is settled, we follow the lines of \cite{Hagger} to obtain our next main result:

\renewcommand*{\themthm}{B}
\begin{mthm} \label{B}
Let $A \in \Lc(A^p_{\nu})$ be band-dominated. Then the following are equivalent:
\begin{itemize}
\item[$(i)$] $A$ is Fredholm,
\item[$(ii)$] $A_x$ is invertible for all $x \in \beta\Omega \setminus \Omega$ and $\sup\limits_{x \in \beta\Omega \setminus \Omega} \norm{A_x^{-1}} < \infty$,
\item[$(iii)$] $A_x$ is invertible for all $x \in \beta\Omega \setminus \Omega$,
\end{itemize}
\end{mthm}

Here, $\beta\Omega$ denotes the Stone-\v{C}ech compactification of $\Omega$ and the operators $A_x$ are the limit operators of $A$. It is worth mentioning that a similar result was also obtained for the Fock space in \cite{FuHa}.

As observed in \cite{Hagger}, the Toeplitz algebra is contained in the set of band-dominated operators (see Section \ref{BDO} and Definition \ref{BDO_on_A^p} for the precise statement) and thus all of the above can be applied to Toeplitz operators. At this point we should probably mention that we use slightly different conventions here than used in previous work. Most importantly, we replaced the compactification of $\Omega$ that labels our limit operators. In \cite{Hagger,MiSuWi,MiWi2,Suarez} and also in \cite{BaIs,FuHa} the maximal ideal space of bounded uniformly continuous functions was used. In this paper we use the Stone-\v{C}ech compactification\footnote{The Stone-\v{C}ech compactification can be seen as the maximal ideal space of bounded continuous functions.} instead to simplify a few arguments. It is then easy to see that one can use smaller compactifications if the operator permits. More precisely, if all limit operators of $A$ with respect to another compactification exist, then the same results hold with the Stone-\v{C}ech compactification replaced by the new compactification. Informally, the more complicated the operator in question behaves towards the boundary, the more complicated compactifications we need. We note that the authors of \cite{MiWi} also used the Stone-\v{C}ech compactification to show convergence of the operator nets, but there was no need to actually label the limit operators at that point. Furthermore, we use a slightly different definition for limit operators here. This change turns out to be merely cosmetic and is certainly just a matter of taste. We refer to Remark \ref{rem_limit_operator_definition} for a short discussion.

Last but not least, we discuss some applications of Theorem \ref{B}. The simplest case one could imagine would be if every limit operator of an operator $A$ was just a multiple of the identity. We show that this happens if and only if the (generalized) Berezin transform of $A$ has vanishing oscillation at the boundary. If this is the case, the essential spectrum of $A$ is equal to the image of the Berezin transform $\Bc(A)$ restricted to the Stone-\v{C}ech boundary of $\Omega$. In particular, this applies to Toeplitz operators whose symbols have vanishing mean oscillation at the boundary:

\begin{cor*}
Let $f \in \VMO_{\partial}(\Omega) \cap L^{\infty}(\Omega)$. Then
\[\spec_{\ess}(T_f) = \tilde{f}(\beta\Omega \setminus \Omega) = \bigcap\limits_{r > 0} \overline{\tilde{f}(\Omega \setminus D(0,r))},\]
where $f$ denotes the Berezin transform of $f$ (or equivalently $T_f$).
\end{cor*}

The paper is organized as follows. We start with a short introduction to bounded symmetric domains in Section \ref{BSDs}. All results in that section are well-known and documented in the literature. In Section \ref{BDO} we recall the definition and some major results of band-dominated operators, which were established in \cite{Hagger} and are then extensively used throughout this paper. Section \ref{limit_operators} then proceeds with the definition of limit operators. In particular, their existence is shown and it will get clear why band-dominated operators are the right objects to consider here. In Section \ref{compactness} we characterize compactness with the help of limit operators and explain the connections to the Berezin transform. We then follow \cite{Hagger} for the characterization of Fredholmness of band-dominated operators in Section \ref{Fredholmness}. In Section \ref{applications} we present some applications to Toeplitz operators.

\section{Bounded Symmetric Domains} \label{BSDs}

In this section we provide definitions and basic properties of bounded symmetric domains. All results in this section are well-known and may be found in the literature (e.g.~\cite{Arazy,Cartan,Englis,Englis2,FaKo,Helgason,Hua,Timoney,Upmeier,Zhu95}).

A domain $\Omega \subset \C^n$ is called a bounded symmetric domain if it is bounded and for every $z \in \Omega$ there exists a biholomorphic involution $\phi_z$ that interchanges $0$ and $z$. A bounded symmetric domain is called irreducible if it is not biholomorphic to a product of two non-trivial domains. The irreducible bounded symmetric domains can be classified as follows\footnote{The order of $II_n$ and $III_n$ is sometimes interchanged. The restrictions on $n$ and $m$ are due to the fact that in small dimensions some of the domains are not irreducible or isomorphic to another domain in the list, e.g.~$I_{1,1}$, $II_1$, $III_2$ and $IV_1$ all describe the unit disk.}:
\begin{itemize}
	\item $I_{n,m}$: unit ball of $n \times m$ complex matrices for $n \geq m \geq 1$
	\item $II_n$: unit ball of $n \times n$ complex symmetric matrices for $n \geq 2$
	\item $III_n$: unit ball of $n \times n$ complex antisymmetric matrices for $n \geq 5$
	\item $IV_n$: the Lie ball $\set{z \in \C^n : \abs{\sum\limits_{j = 1}^n z_j^2} < 1, 1 + \abs{\sum\limits_{j = 1}^n z_j^2}^2 - 2\abs{z}^2 > 0}$ for $n \geq 5$
	\item $V$: the unit ball of $1 \times 2$ matrices over the $8$-dimensional Cayley algebra
	\item $VI$: the unit ball of $3 \times 3$ self-adjoint matrices over the $8$-dimensional Cayley algebra
\end{itemize}
By Cartan's classification theorem, these are all possible cases up to biholomorphisms. We may therefore always assume that $\Omega$ is convex, circular and centered at the origin. This is usually called the Harish-Chandra realization of $\Omega$. Throughout this paper we will assume that $\Omega$ is a irreducible bounded symmetric domain in its Harish-Chandra realization, i.e.~$\Omega$ is equal to one of the cases $I-VI$. Note that case $I_{n,1}$ is isomorphic to the standard unit ball $\B^n$ of $\C^n$.

Let $\Aut(\Omega)$ denote the group of all biholomorphic endomorphisms of $\Omega$ and $G := \Aut(\Omega)_0$ the connected component of $\Aut(\Omega)$ that contains the identity. Moreover, let $K$ denote the subgroup of linear mappings in $G$. By Cartan's linearity theorem, $K$ coincides with the subset of elements in $G$ that stabilize the origin. Therefore $\Omega$ may also be realized as the quotient $G/K$ via $z \mapsto \phi_zK$.

Another description of bounded symmetric domains $\Omega$ can be given in terms of so-called Jordan frames, i.e.~there exists a set of $\R$-linear independent vectors $\set{e_1, \ldots, e_r}$ such that every $z \in \C^n$ can be written as $k\sum\limits_{j = 1}^r t_je_j$ with $k \in K$, $t_j \geq 0$ and it holds $z \in \Omega$ if and only if $t_j < 1$ for all $j \in \set{1, \ldots, r}$. Such a decomposition is called a polar decomposition. The numbers $t_j$ are unique up to permutation and do not depend on the chosen Jordan frame. In particular, the positive integer $r$ is an invariant and called the rank of $\Omega$. In case of the unit ball $\B^n$ the rank is $1$ and $K$ is equal to the full unitary group $U(n)$, i.e.~the polar decomposition is just the usual one with $t_1 = \abs{z}$ and $e_1$ an arbitrary unit vector.

Besides the rank $r$, there are other geometric invariants of $\Omega$. These include the complex dimension $n$, the numbers $a,b$, which have to do with root multiplicities of the Lie algebras associated with $G$ and $K$, and the genus $g := a(r-1)+b+2$. In fact, the triple $(r,a,b)$ determines the bounded symmetric domain $\Omega$ uniquely. However, we do not really need the exact values of these invariants in this paper. We thus refer to \cite{Englis} for a list of numbers.

Using this polar decomposition, we can define the so-called Jordan triple determinant $h \from \C^n \times \C^n \to \C$, which is uniquely determined by the diagonal
\[h(z,z) = \prod\limits_{j = 1}^r (1-t_j^2)\]
and the requirement that $h$ is holomorphic in the first and antiholomorphic in the second argument. In fact, $h = h(z,w)$ is a polynomial in $z$ and $\overline{w}$. This polynomial has a lot of important properties that we will use frequently in this paper. Most of them follow immediately from the definition and/or some complex analysis. Here is a quick summary:
\begin{itemize}
	\item[$(i)$] $\abs{h(z,w)} > 0$ for all $z \in \overline{\Omega}$, $w \in \Omega$
	\item[$(ii)$] $h(z,z) = 0$ for all $z \in \partial\Omega$
	\item[$(iii)$] $h(z,0) = 1$ for all $z \in \C^n$
	\item[$(iv)$] $h(w,z) = \overline{h(z,w)}$ for all $w,z \in \C^n$
	\item[$(v)$] $h(kz,kw) = h(z,w)$ for all $w,z \in \C^n$, $k \in K$
\end{itemize}
Using the Jordan triple determinant, we may define a Riemannian metric on $\Omega$ as follows:
\[g_{ij}(z) := -g\frac{\partial}{\partial z_i}\frac{\partial}{\partial \overline{z}_j} \log h(z,z).\]
The integrated form of this metric will be denoted by $\beta(\cdot,\cdot)$ and is called the Bergman metric on $\Omega$. Note that $\beta$ is unbounded, i.e.~$\beta(0,z) \to \infty$ as $z \to \partial\Omega$. For $\Omega = \B^n$ this metric is given by the usual hyperbolic metric on $\B^n$.

The metric space $(\Omega,\beta)$ satisfies a certain local finiteness condition, which can be formulated as follows. There is a fixed integer $N$ such that for every $t \in (0,1)$ there is a disjoint cover of $\Omega$ by Borel sets $(B_{j,t})_{j \in \N}$ satisfying
\begin{itemize}
	\item for every $z \in \Omega$ the set $\set{j \in \N : \dist_{\beta}(z,B_{j,t}) \leq \frac{1}{t}}$ has at most $N$ elements,
	\item there is a constant $C(t)$ such that $\diam_{\beta}(B_{j,t}) \leq C(t)$ for every $j \in \N$.
\end{itemize}
Setting $\Xi_{j,t,k} := \set{z \in \Omega : \dist_{\beta}(z,B_{j,t}) \leq \frac{k}{3t}}$ for $k = 1,2,3$, we can obtain a subordinate partition of unity consisting of functions $\varphi_{j,t} \from \Omega \to [0,1]$ satisfying
\begin{itemize}
	\item[$(a)$] $\sum\limits_{j = 1}^{\infty} \varphi_{j,t}(z) = 1$ for all $z \in \Omega$,
	\item[$(b)$] $\supp \varphi_{j,t} = \Xi_{j,t,1}$ for all $j \in \N$, $t \in (0,1)$,
	\item[$(c)$] $\abs{\varphi_{j,t}(z) - \varphi_{j,t}(w)} \leq 6Nt\beta(z,w)$ for all $w,z \in \Omega$, $j \in \N$ and $t \in (0,1)$.
\end{itemize}
Similarly, we may obtain functions $\psi_{j,t} \from \Omega \to [0,1]$ satisfying
\begin{itemize}
	\item[$(d)$] $\psi_{j,t}(z) = 1$ for all $z \in \Xi_{j,t,2}$, $j \in \N$ and $t \in (0,1)$,
	\item[$(e)$] $\supp \psi_{j,t} = \Xi_{j,t,3}$ for all $j \in \N$ and $t \in (0,1)$,
	\item[$(f)$] $\abs{\psi_{j,t}(z) - \psi_{j,t}(w)} \leq 3t\beta(z,w)$ for all $w,z \in \Omega$, $j \in \N$ and $t \in (0,1)$.
\end{itemize}
An explicit construction of these functions can be found in \cite[Section 3]{Hagger}. We remark that in the printed version of \cite{Hagger} there is an incorrect reference for the existence of the integer $N$. The correct reference is \cite{CaGo} and was corrected in a later version. The number $N-1$ is called the asymptotic dimension of the metric space and studied in coarse geometry. We refer to \cite{BeDra} for equivalent definitions and an overview of the whole subject, also mentioning the result of \cite{CaGo} that we need here. 

For $z \in \Omega \setminus \set{0}$ we now consider the geodesic reflection in the midpoint between $0$ and $z$ (with respect to $\beta$, of course). This defines an isometric and involutional biholomorphism, which we denote by $\phi_z$. Let us we fix this notation here once and for all. For $z = 0$ we set $\phi_0(w) = -w$. Moreover, for $\nu > -1$ we define the probability measure\footnote{Sometimes the weight is parametrized by $\lambda := \nu+g$ instead.} $v_{\nu}$ as
\[\mathrm{d}v_{\nu}(z) := c_{\nu}h(z,z)^{\nu} \, \mathrm{d}v(z),\]
where $v$ denotes the usual Lebesgue measure on $\C^n$ restricted to $\Omega$ and $c_{\nu}$ is a suitable constant such that $v_{\nu}(\Omega) = 1$. It is worth noting that
\begin{itemize}
	\item[$(vi)$] $\int_{\Omega} h(z,z)^{\nu} \, \mathrm{d}v(z) < \infty$ if and only if $\nu > -1$.
\end{itemize}
The Jordan triple determinant $h$ and the measure $v_{\nu}$ transform with respect to $\phi_z$ as follows:
\begin{itemize}
	\item[$(vii)$] $h(\phi_z(x),\phi_z(y)) = \frac{h(z,z)h(x,y)}{h(x,z)h(z,y)}$ for all $x,y,z \in \Omega$
	\item[$(viii)$] $\mathrm{d}v_{\nu}(\phi_z(w)) = \frac{h(z,z)^{\nu+g}}{\abs{h(w,z)}^{2(\nu+g)}} \, \mathrm{d}v_{\nu}(w)$ for all $w,z \in \Omega$ 
\end{itemize}

For $p \in (1,\infty)$ and $\nu > -1$ let $L^p_{\nu} := L^p(\Omega,v_{\nu})$ denote the usual Lebesgue space of $p$-integrable functions and $A^p_{\nu}$ the closed subspace of holomorphic functions in $L^p_{\nu}$. The orthogonal projection $P_{\nu}$ from $L^2_{\nu}$ onto $A^2_{\nu}$ is given by
\[(P_{\nu}f)(z) = \int_{\Omega} f(w)h(z,w)^{-\nu-g} \, \mathrm{d}v_{\nu}(w)\]
(see \cite[Section 3]{FaKo}). For general $p$ we can try to take the same formula to get a projection onto $A^p_{\nu}$. However, this integral operator, again denoted by $P_{\nu}$, is not bounded in general (see \cite[Theorem II.8]{BeTe}). Let us call $(\alpha,\nu,p) \in \R^2 \times (1,\infty)$ admissible if the following inequalities are satisfied:
\begin{equation} \label{admissible}
p(\alpha + 1) > \nu + 1 + \frac{(r-1)a}{2} > p\frac{(r-1)a}{2}.
\end{equation}
By \cite[Proposition 1]{Hagger}, this implies that $P_{\alpha}$ is a bounded projection from $L^p_{\nu}$ onto $A^p_{\nu}$. Throughout this paper we will always assume that $\nu$ is sufficiently large such that $(\nu,\nu,p)$ is admissible and $P_{\nu}$ is bounded as a consequence. Note that for $\alpha = \nu$ there exist more optimal conditions than \eqref{admissible} (e.g.~\cite[Lemma 9]{Englis}). Even for $\alpha \neq \nu$ one can actually improve these inequalities. However, we refrain from doing that because \eqref{admissible} also assures that we can use the Rudin-Forelli estimates (i.e.~\cite[Proposition 8]{Englis}) appropriately. This will get more clear later on.

For $\alpha = \nu$, we can reformulate the condition \eqref{admissible} as follows:
\begin{equation} \label{admissible2}
1 + \frac{(r-1)a}{2(\nu + 1)} < p < 1 + \frac{2(\nu + 1)}{(r-1)a}.
\end{equation}
Note that $1 + \frac{(r-1)a}{2(\nu + 1)}$ is exactly the dual exponent of $1 + \frac{2(\nu + 1)}{(r-1)a}$. This of course makes sense as $(L^p_{\nu})^* \cong L^q_{\nu}$ and $(A^p_{\nu})^* \cong A^q_{\nu}$ via the usual dual pairing. Therefore $P_{\nu}$ is bounded on $L^p_{\nu}$ if and only if $P_{\nu}^*$ is bounded on $L^q_{\nu}$. But $P_{\nu}^*$ is formally the same as $P_{\nu}$, so that $P_{\nu}$ is bounded on $L^p_{\nu}$ if and only if it is bounded on $L^q_{\nu}$. In that case we can consider Toeplitz operators $T_f := P_{\nu}M_f|_{A^p_{\nu}}$ for bounded functions $f \from \Omega \to \C$ and corresponding multiplication operators $M_f$. The function $f$ is called the symbol of $T_f$ and $M_f$. For every bounded symbol $f$, the corresponding Toeplitz operator is bounded with $\norm{T_f} \leq \norm{P_{\nu}}\norm{f}_{\infty}$. 

We conclude this section with a few notations. The set of all bounded linear operators on a Banach space $X$ will be denoted by $\Lc(X)$. The ideal of compact operators will be denoted by $\Kc(X)$. An operator $A \in \Lc(X)$ is called Fredholm if the coset $A + \Kc(X)$ is invertible in the Calkin algebra $\Lc(X) / \Kc(X)$. The essential spectrum of $A$ is given by $\set{\lambda \in \C : A - \lambda I \text{ is not Fredholm}}$ and denoted by $\spec_{\ess}(A)$. The usual spectrum is denoted by $\spec(A)$. The closed subalgebra of $\Lc(A^p_{\nu})$ generated by all Toeplitz operators with bounded symbol will be denoted by $\Tf_{p,\nu}$. The commutator of two operators $A,B \in \Lc(X)$ will be denoted by $[A,B] = AB-BA$.

\section{Band-dominated operators} \label{BDO}

In \cite{Hagger} band-dominated operators on $L^p_{\nu}$ were defined. We quickly recall the definition and some basic facts in this short section.

\begin{defn} \label{defn_BDO}
(\cite[Definition 6]{Hagger})\\
An operator $A \in \Lc(L^p_{\nu})$ is called a band operator if there exists a positive real number $\omega$ such that $M_fAM_g = 0$ for all $f,g \in L^{\infty}(\Omega)$ with $\dist_{\beta}(\supp f,\supp g) > \omega$. The number
\[\inf\set{\omega \in \R : M_fAM_g = 0 \text{ for all } f,g \in L^{\infty}(\Omega) \text{ with } \dist_{\beta}(\supp f,\supp g) > \omega}\]
is called the band width of $A$. An operator $A \in \Lc(L^p_{\nu})$ is called band-dominated if it is the norm limit of band operators. The set of band-dominated operators will be denoted by $\BDO^p_{\nu}$.
\end{defn}

\begin{prop} \label{prop0}
(\cite[Proposition 13]{Hagger})\\
$\BDO^p_{\nu}$ has the following properties:
\begin{itemize}
	\item[$(i)$] It holds $M_f \in \BDO^p_{\nu}$ for all $f \in L^{\infty}(\Omega)$.
	\item[$(ii)$] $\BDO^p_{\nu}$ is a closed subalgebra of $\Lc(L^p_{\nu})$.
	\item[$(iii)$] If $A \in \BDO^p_{\nu}$ is Fredholm, then every regularizer $B$ of $A$ is again in $\BDO^p_{\nu}$. In particular, $\BDO^p_{\nu}$ is inverse closed.
	\item[$(iv)$] $\BDO^p_{\nu}$ contains $\Kc(L^p_{\nu})$ as a closed two-sided ideal.
	\item[$(v)$] It holds $A \in \BDO^p_{\nu} \Longleftrightarrow A^* \in \BDO^q_{\nu}$ for $\frac{1}{p} = \frac{1}{q} = 1$. In particular, $\BDO^2_{\nu}$ is a $C^*$-algebra.
\end{itemize}
\end{prop}

(Extensions of) Toeplitz operators are particular examples of band-dominated operators as the following proposition shows. This fact is one of our main motivations to study band-dominated operators.

\begin{prop} \label{Projections_band_dominated}
Let $(\alpha,\nu,p)$ be an admissible triple. Then $P_{\alpha} \in \BDO^p_{\nu}$. In particular, $AP_{\alpha} \in \BDO^p_{\nu}$ for $A \in \Tf_{p,\nu}$.
\end{prop}

\begin{proof}
Follows directly from (the proof of) \cite[Theorem 7]{Hagger}.
\end{proof}

The following estimate will be crucial for subsequent results. It is quite remarkable that it holds simultaneously for all band operators of a fixed band width $\omega$.

\begin{lem} \label{lem3}
(\cite[Lemma 12]{Hagger})\\
Let $\omega > 0$ and let $a_{j,t} \from \Omega \to [0,1]$ be measurable functions for $j \in \N$ and $t \in (0,1)$. If 
\[\lim\limits_{t \to 0} \inf\limits_{j \in \N} \dist_{\beta}(a_{j,t}^{-1}(U),a_{j,t}^{-1}(V)) \to \infty\]
for all sets $U,V \subset [0,1]$ with $\dist(U,V) > 0$, then for every $\epsilon > 0$ there exists a $t_0 > 0$ such that for all $t < t_0$ and every band operator of band width at most $\omega$ the estimate
\[\sup\limits_{j \in \N} \norm{[A,M_{a_{j,t}}]} \leq 3\norm{A}\epsilon\]
holds.
\end{lem}

\section{Limit Operators} \label{limit_operators}

Let $\lambda \in \R$. Since $h(w,z) \neq 0$ for all $w,z \in \Omega$ and $\Omega$ is simply connected (\cite[Theorem VIII.7.1]{Helgason}), we can choose a branch of $h(w,z)^{\lambda}$ such that $h(w,z)^{\lambda}$ is holomorphic in $w$ and $\bar{z}$ and $h(w,w)^{\lambda} > 0$ for all $w \in \Omega$. For $z \in \Omega$ we now define $U_z^p \in \Lc(L^p_{\nu})$ by
\[(U_z^pf)(w) = f(\phi_z(w))\frac{h(z,z)^{\frac{\nu+g}{p}}}{h(w,z)^{\frac{2(\nu+g)}{p}}}.\]
A standard computation shows that $U_z^p$ is a surjective isometry with $(U_z^p)^2 = I$. Moreover, $U_z^p$ maps holomorphic functions to holomorphic functions.

\begin{lem} \label{lem2}
The map $z \mapsto \phi_z(w)$ is continuous for all $w \in \Omega$.
\end{lem}

\begin{proof}
Considered as a Riemannian manifold, $\Omega$ is simply connected and has non-positive sectional curvature (see \cite[Theorem V.3.1, Theorem VIII.4.6, Theorem VIII.7.1]{Helgason}). Thus, by Cartan-Hadamard, the exponential map $\exp_w \from T_w\Omega \to \Omega$ is a homeomorphism for every point $w \in \Omega$. Let $z^*$ denote the midpoint of the geodesic connecting $0$ and $z \in \Omega$ (i.e.~the point of reflection of the symmetry $\phi_z$) and let $w \in \Omega$. Then by definition of the exponential map, we have $\phi_z(w) = \exp_w(2\exp_w^{-1}(z^*))$. Similarly, $z^* = \exp_0(\frac{1}{2}\exp_0^{-1}(z))$. It follows that $z \mapsto \phi_z(w)$ is indeed continuous.
\end{proof}

\begin{prop} \label{prop2}
Let $\beta\Omega$ denote the Stone-\v{C}ech compactification of $\Omega$ and $A \in \Lc(L^p_{\nu})$. Consider the map $\Psi_A \from \Omega \to \Lc(L^p_{\nu})$ defined by $\Psi_A(z) := U_z^pAU_z^p$. Then $\Psi_A$ has a weakly continuous extension $\tilde{\Psi}_A \from \beta\Omega \to \Lc(L^p_{\nu})$, i.e.~$\tilde{\Psi}_A$ is weakly continuous and $\tilde{\Psi}_A|_{\Omega} = \Psi_A$.
\end{prop}

\begin{proof}
Let $f \in L^p_{\nu}$ be continuous. Then the map $z \mapsto (U_z^pf)(w)$ is continuous for every $w \in \Omega$ by Lemma \ref{lem2}. As $U_z^p$ is an isometry for every $z \in \Omega$, Scheff\'{e}'s Lemma implies that $z \mapsto U_z^pf$ is continuous in $L^p_{\nu}$. Therefore, as continuous functions are dense in $L^p_{\nu}$, $z \mapsto U_z^p$ is strongly continuous by Banach-Steinhaus. This also implies that $\Psi_A$ is strongly continuous, hence weakly continuous. Moreover, we have $\norm{\Psi_A(z)} = \norm{A}$ for every $z \in \Omega$. As bounded sets are relatively compact in the weak operator topology, $\Psi_A$ admits a weakly continuous extension to $\beta\Omega$.
\end{proof}

We will use the notation
\begin{equation} \label{weak_limit_operators}
A_x := \tilde{\Psi}_A(x) = \wlim\limits_{z \to x} U_z^pAU_z^p,
\end{equation}
where the limit is understood in the sense of nets, and call $A_x$ a limit operator of $A$ if $x \in \beta\Omega \setminus \Omega$.

\begin{rem} \label{rem_limit_operator_definition}
Note that in previous work (\cite{Hagger,MiSuWi,Suarez}) limit operators were defined slightly differently, namely by $A_x = \wlim\limits_{z \to x} U_z^pA(U_z^q|_{A^q_{\nu}})^*$ for $A \in \Lc(A^p_{\nu})$, where $q$ denotes the dual exponent of $p$. There are two small differences to our definition here. First of all, we start by defining limit operators on $L^p_{\nu}$ and then later restrict them to $A^p_{\nu}$ (see Definition \ref{BDO_on_A^p}). As Proposition \ref{prop6} shows, this approach results to the same expression as \eqref{weak_limit_operators}, but with the additional benefit that we can use tools for band-dominated operators in this approach. More importantly, we choose to not use the adjoint in the definition. Fortunately, the two definitions only differ by an invertible operator independent of $A$ (i.e.~the limit of $(U_z^q|_{A^q_{\nu}})^*U_z^p$, see Proposition \ref{prop_T_b_x} below) and they are even the same for $p = 2$. The reason we choose \eqref{weak_limit_operators} as our definition of limit operators is that in this way the behavior under multiplication is somewhat better ($U_z^pU_z^p = I \neq (U_z^q|_{A^ q_{\nu}})^*U_z^p$ for $p \neq 2$). On the other hand, we lose the property $(A_x)^* = (A^*)_x$ and we will see that \eqref{weak_limit_operators} behaves slightly worse under the Berezin transform. As we value multiplication just a little bit more in this paper, we chose to use \eqref{weak_limit_operators} instead of the previous choice. We do not claim that doing it in this way is better in general, though. It is rather a matter of taste and we will come back to this at the end of this section (Proposition \ref{prop_T_b_x}).
\end{rem}

The following properties of limit operators are still intact and follow directly from the properties of $U_z^p$ and/or the weak operator convergence.

\begin{prop} \label{limit_operator_properties}
Let $A,B \in \Lc(L^p_{\nu})$, $\lambda \in \C$ and $x \in \beta\Omega$. Then $(A+B)_x = A_x+B_x$, $(\lambda A)_x = \lambda A_x$ and $\norm{A_x} \leq \norm{A}$. Moreover, if $(A_n)_{n \in \N}$ is a sequence in $\Lc(L^p_{\nu})$ that converges to $A$ in norm, then $(A_n)_x$ converges to $A_x$ in norm.
\end{prop}

Moreover, limit operators of band-dominated operators are again band-dominated. To show this we need the following auxiliary result that is obtained by a direct computation (see e.g.~\cite[proof of Proposition 19]{Hagger}):

\begin{prop} \label{multiplication_shift}
Let $\xi \in L^{\infty}(\Omega)$. Then $U_z^pM_{\xi}U_z^p = M_{\xi \circ \phi_z}$ for all $z \in \Omega$. 
\end{prop}

\begin{prop} \label{prop5}
Let $A \in \BDO^p_{\nu}$. Then $A_x \in \BDO^p_{\nu}$ for all $x \in \beta\Omega$. Moreover, if $A$ is a band operator of band width $\omega$, then all operators $A_x$ are band operators of band width at most $\omega$.
\end{prop}

\begin{proof}
In view of Proposition \ref{prop0} and Proposition \ref{limit_operator_properties}, it suffices to show that limit operators of band operators are again band operators and that the band width does not increase. So assume that $A \in \BDO^p_{\nu}$ is a band operator of band width $\omega$. Then $M_fAM_g = 0$ for all $f,g \in L^{\infty}(\Omega)$ with $\dist_{\beta}(\supp f,\supp g) > \omega$. For $U_z^pAU_z^p$ we observe $M_fU_z^pAU_z^pM_g = U_z^pM_{f \circ \phi_z}AM_{g \circ \phi_z}U_z^p$ and since $\phi_z$ is an isometry, $\dist_{\beta}(\supp(f \circ \phi_z),\supp(g \circ \phi_z)) = \dist_{\beta}(\supp f,\supp g)$. Therefore $U_z^pAU_z^p$ is a band operator of band width $\omega$ for all $z \in \Omega$. This argument directly generalizes to weak limits, i.e.~every $A_x$ is a band operator of band width at most $\omega$.
\end{proof}

\begin{lem} \label{lem4}
Let $A \in \BDO^p_{\nu}$ and assume that $\set{a_{j,t} : j \in \N, t \in (0,1)}$ is a family of measurable functions $a_{j,t} \from \Omega \to [0,1]$ with $\lim\limits_{t \to 0} \inf_{j \in \N} \dist_{\beta}(a_{j,t}^{-1}(U),a_{j,t}^{-1}(V)) \to \infty$ for all sets $U,V \subset [0,1]$ with $\dist(U,V) > 0$. Then
\[\lim\limits_{t \to 0} \sup\limits_{x \in \beta\Omega} \sup\limits_{j \in \N} \norm{[A_x,M_{a_{j,t}}]} = 0.\]
\end{lem}

\begin{proof}
First assume that $A$ is a band operator. Then by Proposition \ref{prop5}, every $A_x$ is again a band operator of the same band width. Since also $\norm{A_x} \leq \norm{A}$ for all $x \in \beta\Omega$ by Proposition \ref{limit_operator_properties}, the assertion follows from Lemma \ref{lem3}. Now let $A \in \BDO^p_{\nu}$ be a general band-dominated operator. Then for every $\epsilon > 0$ there is a band operator $B$ such that $\norm{A-B} < \epsilon$. Using the above and Proposition \ref{limit_operator_properties} again, we get
\[\lim\limits_{t \to 0} \sup\limits_{x \in \beta\Omega} \sup\limits_{j \in \N} \norm{[A_x,M_{a_{j,t}}]} \leq \lim\limits_{t \to 0} \sup\limits_{x \in \beta\Omega} \sup\limits_{j \in \N} \norm{[B_x,M_{a_{j,t}}]} + \lim\limits_{t \to 0} \sup\limits_{x \in \beta\Omega} \sup\limits_{j \in \N} \norm{[A_x-B_x,M_{a_{j,t}}]} < 2\epsilon.\qedhere\]
\end{proof}

\begin{defn} \label{BDO_on_A^p}
We will call an operator $A \in \Lc(A^p_{\nu})$ band-dominated if $AP_{\nu} \in \Lc(L^p_{\nu})$ is band-dominated. Moreover, we define $A_x := (AP_{\nu})_x|_{A^p_{\nu}}$ for every $x \in \beta\Omega$.
\end{defn}

\begin{rem} \label{rem_Toeplitz}
According to Proposition \ref{Projections_band_dominated}, all operators in the Toeplitz algebra $\Tf_{p,\nu}$ are band-dominated. We thus emphasize that all subsequent results about band-dominated operators are of course valid for Toeplitz operators and, more generally, for all operators in the Toeplitz algebra.
\end{rem}

\begin{prop} \label{prop6}
Let $A \in \Lc(A^p_{\nu})$ be band-dominated. Then for every net $(z_{\gamma})$ in $\Omega$ converging to some $x \in \beta\Omega$ the net $U_{z_{\gamma}}^pAU_{z_{\gamma}}^p|_{A^p_{\nu}}$ converges strongly to $A_x$.
\end{prop}

\begin{proof}
Let $f \in A^p_{\nu}$, $B := AP_{\nu} \in \BDO^p_{\nu}$ and fix $t \in (0,1)$. Choose a Lipschitz continuous function $a_t \from \Omega \to [0,1]$ with compact support, Lipschitz constant $t$ and $\norm{f-a_tf} \leq t$. Then
\begin{equation} \label{prop6_eq}
\norm{(B_{z_{\gamma}}-B_x)f} \leq \norm{M_{a_t}(B_{z_{\gamma}}-B_x)f} + \norm{M_{1-a_t}(B_{z_{\gamma}}-B_x)f}.
\end{equation}
As $B_{z_{\gamma}} = U_{z_{\gamma}}^pAP_{\nu}U_{z_{\gamma}}^p$ maps holomorphic functions to holomorphic functions, we have $B_{z_{\gamma}} = P_{\nu}B_{z_{\gamma}}$ for all $\gamma$. By Proposition \ref{prop2}, $B_{z_{\gamma}}$ converges weakly to $B_x$ and as weak limits are unique, this also implies $B_x = P_{\nu}B_x$. As $M_{a_t}P_{\nu}$ is compact by \cite[Propositon 15]{Hagger}, the first term in \eqref{prop6_eq} tends to $0$ as $z_{\gamma} \to x$. For the second term we have
\begin{align*}
\norm{M_{1-a_t}(B_{z_{\gamma}}-B_x)f} &\leq \norm{[M_{1-a_t},B_{z_{\gamma}}-B_x]f} + \norm{(B_{z_{\gamma}}-B_x)M_{1-a_t}f}\\
&\leq \norm{[I - M_{a_t},B_{z_{\gamma}}-B_x]f} + \norm{B_{z_{\gamma}}-B_x}\norm{M_{1-a_t}f}\\
&\leq \norm{[M_{a_t},B_{z_{\gamma}}-B_x]f} + \norm{B_{z_{\gamma}}-B_x}t\\
&\leq 2\sup\limits_{y \in \beta\Omega} \norm{[M_{a_t},B_y]f} + 2\sup\limits_{y \in \beta\Omega} \norm{B_y}t.
\end{align*}
Since $a_t$ is Lipschitz continuous with Lipschitz constant $t$, the first term tends to $0$ as $t \to 0$ by Lemma \ref{lem4}. The second term may be estimated by $2\norm{B}t$ and thus tends to $0$ as well. It follows that the second term in \eqref{prop6_eq} can be made as small as desired. We conclude that $U_{z_{\gamma}}^pAU_{z_{\gamma}}^p|_{A^p_{\nu}} = B_{z_{\gamma}}|_{A^p_{\nu}}$ converges strongly to $B_x|_{A^p_{\nu}} = (AP_{\nu})_x|_{A^p_{\nu}}$.
\end{proof}

Similarly as in Proposition \ref{limit_operator_properties}, the following properties hold for band-dominated operators on $A^p_{\nu}$. Note that due to the strong convergence we additionally have multiplicativity.

\begin{cor} \label{limit_operator_properties2}
Let $A,B \in \Lc(A^p_{\nu})$ be band-dominated and $x \in \beta\Omega$. Then $(A+B)_x = A_x+B_x$, $(AB)_x = A_xB_x$ and $\norm{A_x} \leq \norm{A}$. Moreover, if $(A_n)_{n \in \N}$ is a sequence of band-dominated operators in $\Lc(A^p_{\nu})$ that converges to $A$ in norm, then $(A_n)_x$ converges to $A_x$ in norm.
\end{cor}

We also have the following corollary to Proposition \ref{prop6}.

\begin{cor} \label{cor1}
For all band-dominated $A \in \Lc(A^p_{\nu})$ the map $A_{\bullet} \from \beta\Omega \to \Lc(A^p_{\nu})$, $x \mapsto A_x$ is bounded and continuous with respect to the strong operator topology. In particular, the two sets $\set{A_x: x \in \beta\Omega}$ and $\set{A_x : x \in \beta\Omega \setminus \Omega}$ are strongly compact.
\end{cor}

We conclude this section with an observation that will allow us to use some duality arguments.

\begin{prop} \label{prop_T_b_x}
Let $\frac{1}{p} + \frac{1}{q} = 1$. Then $T_{b_z} := (U_z^q|_{A^q_{\nu}})^*U_z^p|_{A^p_{\nu}} \in \Lc(A^p_{\nu})$ is invertible for every $z \in \Omega$. Moreover, if $(z_{\gamma})$ is a net in $\Omega$ that converges to some $x \in \beta\Omega$, then $T_{b_{z_{\gamma}}}$ converges strongly to an invertible operator $T_{b_x}$ and the inverses $T_{b_{z_{\gamma}}}^{-1}$ converge strongly to the inverse $T_{b_x}^{-1}$.
\end{prop}

\begin{proof}
Let $z \in \Omega$. As $U_z^p$ is an isometry, its restriction $U_z^p|_{A^p_{\nu}}$ is an isometry, too. Since $(U_z^p)^2 = I$ and $U_z^p(A^p_{\nu}) \subseteq A^p_{\nu}$, we get $U_z^p(A^p_{\nu}) = A^p_{\nu}$, i.e.~$U_z^p|_{A^p_{\nu}}$ is surjective. In particular, $T_{b_z} = (U_z^q|_{A^q_{\nu}})^*U_z^p|_{A^p_{\nu}}$ is invertible with
\[T_{b_z}^{-1} = U_z^p(U_z^q|_{A^q_{\nu}})^*.\]
In fact, we may also compute $(U_z^q|_{A^q_{\nu}})^*$ explicitly. Let $f \in A^p_{\nu} \cong (A^q_{\nu})^*$ and $g \in A^q_{\nu}$. Then, via the usual dual pairing and the standard transformation formulas, we obtain
\begin{align*}
((U_z^q)^*f)(g) &= f(U_z^q g)\\
&= \int_{\Omega} \overline{f(w)}g(\phi_z(w))\frac{h(z,z)^{\frac{\nu+g}{q}}}{h(w,z)^{\frac{2(\nu+g)}{q}}} \, \mathrm{d}v_{\nu}(w)\\
&= \int_{\Omega} \overline{f(\phi_z(y))}g(y)\frac{h(z,z)^{\frac{\nu+g}{q}}}{h(\phi_z(y),z)^{\frac{2(\nu+g)}{q}}} \frac{h(z,z)^{\nu+g}}{\abs{h(y,z)}^{2(\nu+g)}} \, \mathrm{d}v_{\nu}(y)\\
&= \int_{\Omega} \overline{f(\phi_z(y))}g(y)\frac{h(y,z)^{\frac{2(\nu+g)}{q}}}{h(z,z)^{\frac{\nu+g}{q}}} \frac{h(z,z)^{\nu+g}}{\abs{h(y,z)}^{2(\nu+g)}} \, \mathrm{d}v_{\nu}(y)\\
&= \int_{\Omega} \frac{h(y,z)^{(\frac{2}{q}-1)(\nu+g)}}{h(z,y)^{(1-\frac{2}{p})(\nu+g)}}\overline{f(\phi_z(y))}\frac{h(z,z)^{\frac{\nu+g}{p}}}{h(z,y)^{\frac{2(\nu+g)}{p}}} g(y) \, \mathrm{d}v_{\nu}(y)\\
&= \int_{\Omega} \frac{h(y,z)^{(\frac{1}{q}-\frac{1}{p})(\nu+g)}}{h(z,y)^{(\frac{1}{q}-\frac{1}{p})(\nu+g)}}\overline{(U_z^pf)(y)}g(y) \, \mathrm{d}v_{\nu}(y).
\end{align*}
This implies $(U_z^q|_{A^q_{\nu}})^* = P_{\nu}(U_z^q)^* = T_{b_z}U_z^p|_{A^p_{\nu}}$ with
\[b_z(y) := \frac{h(z,y)^{(\frac{1}{q}-\frac{1}{p})(\nu+g)}}{h(y,z)^{(\frac{1}{q}-\frac{1}{p})(\nu+g)}},\]
which also explains the notation $T_{b_z}$, i.e.~$T_{b_z}$ is a Toeplitz operator with symbol $b_z$.

Now let $(z_{\gamma})$ be a net in $\Omega$ that converges to some $x \in \beta\Omega \setminus \Omega$. Clearly, there is a subnet that converges to some point $\alpha \in \partial\Omega$ with respect to the Euclidean topology on $\overline{\Omega}$. Assume that there is another subnet that converges to a different point $\beta \in \partial\Omega$. This would imply that every continuous function on $\overline{\Omega}$ would coincide in $\alpha$ and $\beta$. As this is not the case, the whole net $(z_{\gamma})$ has to converge to the point $\alpha$.

As $F(z,y) := b_z(y)$ extends to a continuous function on $\overline{\Omega} \times \Omega$, we obtain that $b_{z_{\gamma}}$ converges uniformly on compact sets to $b_{\alpha} = F(\alpha,\cdot)$. Therefore $T_{b_{z_{\gamma}}}$ converges strongly to $T_{b_{\alpha}} =: T_{b_x}$ as $z_{\gamma} \to x$. Similarly, $T_{b_{z_{\gamma}}}^* = T_{\overline{b}_{z_{\gamma}}}$ converges strongly to $T_{\overline{b}_x} = T_{b_x}^* \in \Lc(A^q_{\nu})$. Moreover, $\norm{T_{b_z}^{-1}} \leq \norm{P_{\nu}}$ implies $\norm{T_{b_z}f} \geq \norm{P_{\nu}}^{-1}\norm{f}$ for all $f \in A^p_{\nu}$ and $z \in \Omega$. Similarly, $\norm{T_{\overline{b}_z}g} \geq \norm{P_{\nu}}^{-1}\norm{g}$ for all $g \in A^q_{\nu}$ and $z \in \Omega$. Taking the limit yields $\norm{T_{b_x}f} \geq \norm{P_{\nu}}^{-1}\norm{f}$ and $\norm{T_{b_x}^*g} \geq \norm{P_{\nu}}^{-1}\norm{g}$ for all $f \in A^p_{\nu}$, $g \in A^q_{\nu}$, which implies that $T_{b_x}$ is again invertible. $T_{b_{z_{\gamma}}}^{-1} \to T_{b_x}^{-1}$ strongly as $z_{\gamma} \to x$ follows easily as well.
\end{proof}

As a corollary we obtain the following important result.

\begin{cor} \label{cor_adjoints}
Let $A \in \Lc(A^p_{\nu})$ be band-dominated and $\frac{1}{p} + \frac{1}{q} = 1$. Then for every net $(z_{\gamma})$ in $\Omega$ converging to some $x \in \beta\Omega$ the net $U_{z_{\gamma}}^qA^*U_{z_{\gamma}}^q|_{A^q_{\nu}}$ converges strongly to $T_{\overline{b}_x}^{-1}(A_x)^*T_{\overline{b}_x}$. In particular, $(A^*)_x$ is invertible if and only if $A_x$ is.
\end{cor}

\begin{proof}
As $A^*P_{\nu} = (AP_{\nu})^*$ (with the usual identification of $(A^p_{\nu})^*$ and $A^q_{\nu}$), Proposition \ref{prop0} implies that $A^*$ is band-dominated. By Proposition \ref{prop6}, $U_{z_{\gamma}}^qA^*U_{z_{\gamma}}^q|_{A^q_{\nu}}$ converges strongly to $(A^*)_x$. On the other hand,
\[\left(U_{z_{\gamma}}^qA^*U_{z_{\gamma}}^q|_{A^q_{\nu}}\right)^* = (U_{z_{\gamma}}^q|_{A^q_{\nu}})^*A(U_{z_{\gamma}}^q|_{A^q_{\nu}})^* = T_{b_{z_{\gamma}}}U_{z_{\gamma}}^pAU_{z_{\gamma}}^pT_{b_{z_{\gamma}}}^{-1}\]
converges strongly to $T_{b_x}A_xT_{b_x}^{-1}$ by Proposition \ref{prop_T_b_x}. As strong limits are unique, this implies $((A^*)_x)^* = T_{b_x}A_xT_{b_x}^{-1}$, or equivalently, $(A^*)_x = T_{\overline{b}_x}^{-1}(A_x)^*T_{\overline{b}_x}$.
\end{proof}

\section{Compactness} \label{compactness}

Before we proceed with the characterization of Fredholmness in terms of limit operators, we need to characterize compactness. The next proposition shows that all limit operators of compact operators vanish.

\begin{prop} \label{compact_limit_operators}
Let $K \in \Lc(A^p_{\nu})$ be compact. Then $K$ is band-dominated and $K_x = 0$ for all $x \in \beta\Omega \setminus \Omega$.
\end{prop}

\begin{proof}
That $K$ is band-dominated follows immediately from Proposition \ref{prop0}. Proposition \ref{prop6} thus implies that $U_{z_{\gamma}}^pKU_{z_{\gamma}}^p|_{A^p_{\nu}}$ converges strongly to $K_x$ for every net $(z_{\gamma})$ that converges to $x \in \beta\Omega \setminus \Omega$. Fix $f \in A^p_{\nu}$ and set $D(w,R) := \set{z \in \Omega : \beta(w,z) < R}$ for $R > 0$ and $w \in \Omega$. Proposition \ref{multiplication_shift} and the fact that $U_{z_{\gamma}}^p$ is an isometry imply
\begin{align*}
\norm{U_{z_{\gamma}}^pKU_{z_{\gamma}}^pf} &\leq \norm{U_{z_{\gamma}}^pKP_{\nu}M_{\chi_{D(0,R)}}U_{z_{\gamma}}^pf} + \norm{U_{z_{\gamma}}^pKP_{\nu}M_{1-\chi_{D(0,R)}}U_{z_{\gamma}}^pf}\\
&\leq \norm{KP_{\nu}}\norm{M_{\chi_{D(0,R)}}U_{z_{\gamma}}^pf} + \norm{KP_{\nu}M_{1-\chi_{D(0,R)}}}\norm{f}\\
&= \norm{KP_{\nu}}\norm{M_{\chi_{D(z_{\gamma},R)}}f} + \norm{KP_{\nu}M_{1-\chi_{D(0,R)}}}\norm{f}.
\end{align*}
As $\chi_{D(z_{\gamma},R)}$ converges pointwise to $0$, the first term tends to $0$ for every fixed $R > 0$ as $z_{\gamma} \to x$. On the other hand, $1-\chi_{D(0,R)}$ converges pointwise to $0$ as $R \to \infty$. Therefore $M_{1-\chi_{D(0,R)}}$ converges strongly to $0$ and since $K$ is compact, $KP_{\nu}M_{1-\chi_{D(0,R)}}$ tends to $0$ in norm as $R \to \infty$. Therefore, if $R$ is chosen sufficiently large, the second term can be made as small as desired. We thus conclude $\norm{U_{z_{\gamma}}^pKU_{z_{\gamma}}^pf} \to 0$ as $z_{\gamma} \to x$. As this is true for every $f \in A^p_{\nu}$, $K_x = 0$ follows.
\end{proof}

Our goal for this section is to show that the converse is true as well, i.e.~if $K \in \Lc(A^p_{\nu})$ is band-dominated and $K_x = 0$ for all $x \in \beta\Omega \setminus \Omega$, then $K$ must be compact. For this we need a few auxiliary results.

\begin{prop} \label{interchange_prop}
Let $\alpha = (\frac{2}{p} - 1)g + \frac{2\nu}{p}$ and assume that $(\alpha,\nu,p)$ is admissible. Then $P_{\alpha}U_z^p = U_z^pP_{\alpha}$.
\end{prop}

\begin{proof}
Using the usual transformation formulas, we get
\begin{align}
(P_{\alpha}U_z^pf)(x) &= \int_{\Omega} f(\phi_z(w))\frac{h(z,z)^{\frac{\nu+g}{p}}}{h(w,z)^{\frac{2(\nu+g)}{p}}} h(x,w)^{-\alpha-g}\, \mathrm{d}v_{\alpha}(w)\notag\\
&= \int_{\Omega} f(y)\frac{h(z,z)^{\frac{\nu+g}{p}}}{h(\phi_z(y),z)^{\frac{2(\nu+g)}{p}}} h(x,\phi_z(y))^{-\alpha-g} \frac{h(z,z)^{\alpha+g}}{\abs{h(y,z)}^{2(\alpha+g)}} \, \mathrm{d}v_{\alpha}(y)\notag\\
&= \int_{\Omega} f(y)\frac{h(z,z)^{\frac{\nu+g}{p}}h(y,z)^{\frac{2(\nu+g)}{p}}}{h(z,z)^{\frac{2(\nu+g)}{p}}} \frac{h(x,z)^{-\alpha-g}h(\phi_z(x),y)^{-\alpha-g}}{h(z,y)^{-\alpha-g}} \frac{h(z,z)^{\alpha+g}}{\abs{h(y,z)}^{2(\alpha+g)}} \, \mathrm{d}v_{\alpha}(y)\notag\\
&= \frac{h(z,z)^{-\frac{\nu+g}{p}+\alpha+g}}{h(x,z)^{\alpha+g}}\int_{\Omega} f(y)h(y,z)^{\frac{2(\nu+g)}{p}-\alpha-g} h(\phi_z(x),y)^{-\alpha-g} \, \mathrm{d}v_{\alpha}(y)\notag\\
&= \frac{h(z,z)^{\frac{\nu+g}{p}}}{h(x,z)^{\frac{2(\nu+g)}{p}}}\int_{\Omega} f(y)h(\phi_z(x),y)^{-\alpha-g} \, \mathrm{d}v_{\alpha}(y)\notag\\
&= (U_z^pP_{\alpha}f)(x).\qedhere
\end{align}
\end{proof}

Let $r_t := \sup\limits_{j \in \N} \diam_{\beta}\supp \varphi_{j,t}$ for $t \in (0,1)$. By property $(b)$ of the functions $\varphi_{j,t}$, $r_t$ is finite for all $t$. Similarly as in \cite{Hagger,HaLiSe}, we define
\[\vertiii{A|_F}_t := \sup\set{\norm{Af} : f \in L^p_{\nu}, \norm{f} = 1, \supp f \subseteq D(w,r_t) \cap F \text{ for some } w \in \Omega}\]
and
\[\norm{A|_F} := \sup\set{\norm{Af} : f \in L^p_{\nu}, \norm{f} = 1, \supp f \subseteq F}\]
for $t \in (0,1)$, $A \in \Lc(L^p_{\nu})$ and a Borel set $F \subseteq \Omega$.

\begin{prop} \label{prop9}
Let $A \in \BDO^p_{\nu}$. Then for every $\epsilon > 0$ there exists a $t \in (0,1)$ such that for all Borel sets $F \subseteq \Omega$ and all operators $B$ in the set
\[\{A\} \cup \set{A_x : x \in \beta\Omega \setminus \Omega}\]
it holds
\[\|B|_F\| \geq \vertiiis{B|_F}_t \geq \|B|_F\| - \epsilon.\]
\end{prop}

\begin{proof}
The first inequality is clear by definition. For the second inequality we first assume that $A$ is a band operator. Then, by Proposition \ref{prop5}, all limit operators $A_x$ have the same band width as $A$. Let $B \in \{A\} \cup \set{A_x : x \in \beta\Omega \setminus \Omega}$, $F \subseteq \Omega$ a Borel set and choose $f \in L^p_{\nu}$ with $\norm{f} = 1$ and $\supp f \subseteq F$ such that
\[\|Bf\| \geq \|B|_F\| - \frac{\epsilon}{2}.\]
Moreover, let $\varphi_{j,t}$ and $\psi_{j,t}$ be as defined in Section \ref{BSDs}. Then
\begin{align*}
\left(\sum\limits_{j = 1}^{\infty} \norm{BM_{\varphi_{j,t}^{1/p}}f}^p\right)^{1/p} &= \left(\sum\limits_{j = 1}^{\infty} \norm{BM_{\varphi_{j,t}^{1/p}}M_{\psi_{j,t}}f}^p\right)^{1/p}\\
&\geq \left(\sum\limits_{j = 1}^{\infty} \norm{M_{\varphi_{j,t}^{1/p}}Bf}^p\right)^{1/p} - \left(\sum\limits_{j = 1}^{\infty} \norm{M_{\varphi_{j,t}^{1/p}}BM_{1-\psi_{j,t}}f}^p\right)^{1/p}\\
&\qquad - \left(\sum\limits_{j = 1}^{\infty} \norm{[B,M_{\varphi_{j,t}^{1/p}}]M_{\psi_{j,t}}f}^p\right)^{1/p}
\end{align*}
by Minkowski's inequality. The first term is exactly $\|Bf\|$ since $\sum\limits_{j = 1}^{\infty} \abs{\varphi_{j,t}(z)} = 1$. The second term vanishes if $\dist_{\beta}(\supp \varphi_{j,t},\supp(1 - \psi_{j,t})) = \frac{2}{3t}$ exceeds the band width of $A$. The third term can be estimated as
\begin{align*}
\left(\sum\limits_{j = 1}^{\infty} \norm{[B,M_{\varphi_{j,t}^{1/p}}]M_{\psi_{j,t}}f}^p\right)^{1/p} &\leq \sup\limits_{j \in \N} \norm{[B,M_{\varphi_{j,t}^{1/p}}]}\left(\sum\limits_{j = 1}^{\infty} \norm{M_{\psi_{j,t}}f}^p\right)^{1/p}\notag\\
&\leq N^{1/p}\sup\limits_{j \in \N} \norm{[B,M_{\varphi_{j,t}^{1/p}}]}
\end{align*}
because every $z \in \Omega$ is contained in at most $N$ sets $\supp \psi_{j,t}$ and $\norm{f} = 1$. Now observe that the functions $\varphi_{j,t}^{1/p}$ satisfy the assumptions in Lemma \ref{lem3}. Indeed, let $U,V \subset [0,1]$ with $\dist(U,V) > 0$ and $w_{j,t} \in (\varphi_{j,t}^{1/p})^{-1}(U)$, $z_{j,t} \in (\varphi_{j,t}^{1/p})^{-1}(V)$. Clearly, we have $\dist(U^p,V^p) > 0$ as well and therefore
\[\beta(z_{j,t},w_{j,t}) \geq \frac{1}{6Nt}\abs{\varphi_{j,t}(z_{j,t}) - \varphi_{j,t}(w_{j,t})} \geq \frac{1}{6Nt}\dist(U^p,V^p) \to \infty\]
as $t \to 0$. Lemma \ref{lem3} thus implies that for every $\delta > 0$ there is a $t > 0$ such that
\[\left(\sum\limits_{j = 1}^{\infty} \norm{[B,M_{\varphi_{j,t}^{1/p}}]M_{\psi_{j,t}}f}^p\right)^{1/p} \leq \delta\norm{B} \leq \delta\norm{A}.\]
We thus choose $\delta = \frac{\epsilon}{2\norm{A}}$ and obtain
\[\left(\sum\limits_{j = 1}^{\infty} \norm{BM_{\varphi_{j,t}^{1/p}}f}^p\right)^{1/p} \geq  \|Bf\| - \frac{\epsilon}{2} \geq \|B|_F\| - \epsilon = \left(\|B|_F\| - \epsilon\right)\left(\sum\limits_{j = 1}^{\infty} \norm{M_{\varphi_{j,t}^{1/p}}f}^p\right)^{1/p}.\]
This implies, in particular, that there exists a $j \in \N$ such that
\[\norm{BM_{\varphi_{j,t}^{1/p}}f} \geq \left(\|B|_F\| - \epsilon\right)\norm{M_{\varphi_{j,t}^{1/p}}f}\]
for sufficiently small $t$. Since $\supp \left(M_{\varphi_{j,t}^{1/p}}f\right) \subseteq \supp \varphi_{j,t} \subseteq D(w,r_t)$ for some $w \in \Omega$ by definition, this implies $\vertiiis{B|_F}_t \geq \|B|_F\| - \epsilon$ for all $B \in \{A\} \cup \set{A_x : x \in \beta\Omega \setminus \Omega}$. As $t$ is chosen independently of $F$ (as it is chosen independently of $f$) and $B$, the assertion follows for band operators $A$.

For general band dominated operators the result follows by approximation. Just observe that
\begin{align*}
\vertiiis{(A-A_n)|_F}_t &\leq \|(A-A_n)|_F\| \leq \norm{A-A_n} \quad \text{and}\\
\vertiiis{(A_x-(A_n)_x)|_F}_t &\leq \|(A_x-(A_n)_x)|_F\| \leq \norm{A-A_n}.\qedhere
\end{align*}
\end{proof}

The next theorem now shows that $\sup\limits_{x \in \beta\Omega \setminus \Omega} \norm{A_x}$ is equivalent to $\norm{A + \Kc(A^p_{\nu})}$, the quotient norm of $A + \Kc(A^p_{\nu}) \in \Lc(A^p_{\nu}) / \Kc(A^p_{\nu})$, for all band-dominated operators.

\begin{thm} \label{thm5}
Let $A \in \Lc(A^p_{\nu})$ be band-dominated. Then
\[\frac{1}{\norm{P_{\nu}}}\norm{A + \Kc(A^p_{\nu})} \leq \sup\limits_{x \in \beta\Omega \setminus \Omega} \norm{A_x} \leq \norm{A + \Kc(A^p_{\nu})}.\]
In particular, $K \in \Lc(A^p_{\nu})$ is compact if and only if $K$ is band-dominated and $K_x = 0$ for all $x \in \beta\Omega \setminus \Omega$.
\end{thm}

\begin{proof}
Let $x \in \beta\Omega \setminus \Omega$, $K \in \Kc(A^p_{\nu})$ and choose a net $(z_{\gamma})$ in $\Omega$ that converges to $x$. As $K$ is compact, we get $K_x = 0$ by Proposition \ref{compact_limit_operators}. Corollary \ref{limit_operator_properties2} thus implies
\[\norm{A_x} = \norm{A_x+K_x} = \norm{(A+K)_x} \leq \norm{A+K}.\]
As this is true for all $K \in \Kc(A^p_{\nu})$ and $x \in \beta\Omega \setminus \Omega$, the second inequality follows.

For the first inequality we observe that
\[\norm{AP_{\nu} + K} = \sup\limits_{\norm{f} = 1} \norm{(AP_{\nu} + K)f} \geq \sup\limits_{\substack{f \in A^p_{\nu},\\ \norm{f} = 1}} \norm{(AP_{\nu} + K)f} = \sup\limits_{\substack{f \in A^p_{\nu},\\ \norm{f} = 1}} \norm{(A + K)f} = \norm{A+K|_{A^p_{\nu}}}\]
for all compact operators $K \from L^p_{\nu} \to A^p_{\nu}$. Thus
\[\norm{A + \Kc(A^p_{\nu})} \leq \inf\limits_{K \in \Kc(L^p_{\nu},A^p_{\nu})}\norm{AP_{\nu} + K}.\]
We will now show
\begin{equation} \label{norm_inequality}
\inf\limits_{K \in \Kc(L^p_{\nu},A^p_{\nu})}\norm{AP_{\nu} + K} \leq \sup\limits_{x \in \beta\Omega \setminus \Omega} \norm{A_xT_{b_x}^{-1}P_{\nu}}.
\end{equation}
This will imply the desired inequality since
\[\norm{T_{b_x}^{-1}P_{\nu}} \leq \sup\limits_{z \in \Omega} \norm{U_z^pP_{\nu}(U_z^q)^*P_{\nu}} = \sup\limits_{z \in \Omega} \norm{U_z^p(P_{\nu}U_z^qP_{\nu})^*} = \sup\limits_{z \in \Omega} \norm{U_z^p(U_z^qP_{\nu})^*} \leq \norm{P_{\nu}},\]
where we used that $U_z^p$ and $U_z^q$ are isometries and $U_z^q(A^q_{\nu}) \subseteq A^q_{\nu}$. So assume that \eqref{norm_inequality} is violated, i.e.~that there is an $\epsilon > 0$ such that
\[\inf\limits_{K \in \Kc(L^p_{\nu},A^p_{\nu})}\norm{AP_{\nu} + K} > \sup\limits_{x \in \beta\Omega \setminus \Omega} \norm{A_xT_{b_x}^{-1}P_{\nu}} + \epsilon.\]
In particular,
\[\norm{AP_{\nu}|_{\Omega \setminus D(0,s)}} = \norm{AP_{\nu}M_{1-\chi_{D(0,s)}}} = \norm{AP_{\nu} - AP_{\nu}M_{\chi_{D(0,s)}}} > \sup\limits_{x \in \beta\Omega \setminus \Omega} \norm{A_xT_{b_x}^{-1}P_{\nu}} + \epsilon\]
for all $s > 0$ since $P_{\nu}M_{\chi_{D(0,s)}} \in \Kc(L^p_{\nu},A^p_{\nu})$ (see e.g.~\cite[Proposition 15]{Hagger}). Now, by Proposition \ref{prop9}, there is a $t \in (0,1)$ such that for all $s > 0$ we have
\[\vertiii{AP_{\nu}|_{\Omega \setminus D(0,s)}}_t \geq \norm{AP_{\nu}|_{\Omega \setminus D(0,s)}} - \frac{\epsilon}{2} > \sup\limits_{x \in \beta\Omega \setminus \Omega} \norm{A_xT_{b_x}^{-1}P_{\nu}} + \frac{\epsilon}{2}.\]
In particular, for every $s > 0$ we get a $w_s \in \Omega$ such that
\[\norm{AP_{\nu}M_{\chi_{D(w_s,r_t)}}} \geq \norm{AP_{\nu}M_{\chi_{D(w_s,r_t) \setminus D(0,s)}}} > \sup\limits_{x \in \beta\Omega \setminus \Omega} \norm{A_xT_{b_x}^{-1}P_{\nu}} + \frac{\epsilon}{2}.\]
It is clear that $w_s \to \partial\Omega$ as $s \to \infty$ (otherwise we would get $0$ at some point in the middle term). Moreover,
\begin{align*}
\norm{U_{w_s}^pA(U_{w_s}^q|_{A^q_{\nu}})^*P_{\nu}M_{\chi_{D(0,r_t)}}} &= \norm{U_{w_s}^pA(P_{\nu}U_{w_s}^qP_{\nu})^*M_{\chi_{D(0,r_t)}}} = \norm{U_{w_s}^pA(U_{w_s}^qP_{\nu})^*M_{\chi_{D(0,r_t)}}}\\
&= \norm{AP_{\nu}(U_{w_s}^q)^*M_{\chi_{D(0,r_t)}}(U_{w_s}^q)^*} =  \norm{AP_{\nu}M_{\chi_{D(w_s,r_t)}}}\\
&> \sup\limits_{x \in \beta\Omega \setminus \Omega} \norm{A_xT_{b_x}^{-1}P_{\nu}} + \frac{\epsilon}{2},
\end{align*}
where we used the fact that both $U_{w_s}^p$ and $U_{w_s}^q$ are surjective isometries, $U_z^q(A^q_{\nu}) \subseteq A^q_{\nu}$ and $(U_{w_s}^q)^*M_{\chi_{D(0,r_t)}}(U_{w_s}^q)^* = M_{\chi_{D(w_s,r_t)}}$ (cf.~Proposition \ref{multiplication_shift}).

As $\beta\Omega$ is compact, $(w_s)$ has a convergent subnet, again denoted by $(w_s)$, converging to some $y \in \beta\Omega \setminus \Omega$. Proposition \ref{prop6} and Proposition \ref{prop_T_b_x} imply $U_{w_s}^pAU_{w_s}^p|_{A^p_{\nu}} \to A_y$ and $T_{b_{w_s}}^{-1} \to T_{b_y}^{-1}$ strongly and hence
\[\norm{U_{w_s}^pA(U_{w_s}^q|_{A^q_{\nu}})^*P_{\nu}M_{\chi_{D(0,r_t)}}} = \norm{U_{w_s}^pAU_{w_s}^pT_{b_{w_s}}^{-1}P_{\nu}M_{\chi_{D(0,r_t)}}} \to \norm{A_yT_{b_y}^{-1}P_{\nu}M_{\chi_{D(0,r_t)}}}\]
since $P_{\nu}M_{\chi_{D(0,r_t)}}$ is compact. This yields
\[\norm{A_yT_{b_y}^{-1}P_{\nu}M_{\chi_{D(0,r_t)}}} \geq \sup\limits_{x \in \beta\Omega \setminus \Omega} \norm{A_xT_{b_x}^{-1}P_{\nu}} + \frac{\epsilon}{2},\]
which is certainly a contradiction. Thus $\inf\limits_{K \in \Kc(L^p_{\nu},A^p_{\nu})}\norm{AP_{\nu} + K} \leq \sup\limits_{x \in \beta\Omega \setminus \Omega} \norm{A_xT_{b_x}^{-1}P_{\nu}}$ and the theorem follows as mentioned above.
\end{proof}

In \cite{MiSuWi,Suarez} the unit ball $\B^n$ was considered and compactness was characterized in terms of the Berezin transform. Using Theorem \ref{thm5}, we can generalize this characterization to bounded symmetric domains $\Omega$. In \cite{MiWi} a similar result was obtained for what the authors call ``Bergman-type spaces'' in case $p = 2$.

For $\frac{1}{p} + \frac{1}{q} = 1$ we define
\[k_z^{(p)}(w) := \frac{h(z,z)^{\frac{\nu+g}{q}}}{h(w,z)^{\nu+g}}.\]
For $p = 2$ this function is called the normalized reproducing kernel. A quick computation using the Rudin-Forelli estimates \cite[Proposition 8]{Englis} shows that $k_z^{(p)}$ is contained in $A^p_{\nu}$ and that $C_p := \sup\limits_{z \in \Omega} \|k_z^{(p)}\|$ is finite if $(\nu,\nu,p)$ is admissible, which, as already mentioned a few times, is assumed throughout this paper (see Section \ref{BSDs}). As $(\nu,\nu,p)$ is admissible if and only if  $(\nu,\nu,q)$ is admissible, this also implies $C_q := \sup\limits_{z \in \Omega} \|k_z^{(q)}\| < \infty$ (see Equation \ref{admissible2}). We may thus define the Berezin transform $\Bc(A) \from \Omega \to \C$ of an operator $A \in \Lc(A^p_{\nu})$ as
\[(\Bc(A))(z) := \int_{\Omega} (Ak_z^{(p)})(w)\overline{k_z^{(q)}(w)} \, \mathrm{d}v_{\nu}(w).\]
Using H\"older's inequality, it is not difficult to see that $\Bc(A)$ is bounded and uniformly continuous with respect to the Bergman metric. Moreover, we have $\Bc(A) = 0$ if and only if $A = 0$ by standard arguments (see e.g.~\cite[Section 2]{Stroethoff} or \cite[Section 7.2]{Zhu95}).

\renewcommand*{\themthm}{A}
\begin{mthm} \label{thm6}
An operator $K \in \Lc(A^p_{\nu})$ is compact if and only if $K$ is band-dominated and
\[\lim\limits_{z \to \partial\Omega} (\Bc(K))(z) = 0.\]
\end{mthm}

\begin{proof}
We will show that $K_x = 0$ for all $x \in \beta\Omega \setminus \Omega$ if and only if $\lim\limits_{z \to \partial\Omega} (\Bc(K))(z) = 0$. The result then follows by Proposition \ref{compact_limit_operators} and Theorem \ref{thm5}.

Choose a net $(z_{\gamma})$ in $\Omega$ that converges to some $x \in \beta\Omega \setminus \Omega$ and let $K_x = 0$. Consider the functions $f_z \in A^p_{\nu}$ defined by $f_z(w) := h(w,z)^{(\nu+g)(1-\frac{2}{p})}$ for $w,z \in \Omega$. It holds
\[(U_z^pf_z)(w) = h(\phi_z(w),z)^{(\nu+g)(1-\frac{2}{p})} \frac{h(z,z)^{\frac{\nu+g}{p}}}{h(w,z)^{\frac{2(\nu+g)}{p}}} = \frac{h(z,z)^{(\nu+g)(1-\frac{2}{p})}}{h(w,z)^{(\nu+g)(1-\frac{2}{p})}} \frac{h(z,z)^{\frac{\nu+g}{p}}}{h(w,z)^{\frac{2(\nu+g)}{p}}} = \frac{h(z,z)^{\frac{\nu+g}{q}}}{h(w,z)^{\nu+g}},\]
i.e.~$U_z^pf_z = k_z^{(p)}$. In particular, we have $\sup\limits_{z \in \Omega} \|f_z\| \leq C_p$ because $U_z^p$ is an isometry for all $z \in \Omega$. Moreover,
\[(\Bc(K))(z) \leq \|Kk_z^{(p)}\|_p\|k_z^{(q)}\|_q = \|U_z^pKU_z^pf_z\|_p\|k_z^{(q)}\|_q \leq C_q\|U_z^pKU_z^pf_z\|_p.\]
As in the proof of Proposition \ref{prop_T_b_x}, $(z_{\gamma})$ converges to some $\alpha \in \partial\Omega$ in the Euclidean topology. Using that $h(w,z)$ is a polynomial in $w$ and $\overline{z}$ and $\abs{h(w,z)} > 0$ on $\Omega \times \overline{\Omega}$, we get that $f_{z_{\gamma}}$ converges uniformly on compact sets to a bounded function $f_{\alpha}$. In particular, $f_{z_{\gamma}} \to f_{\alpha}$ in $A^p_{\nu}$ and $f_{\alpha} \in A^p_{\nu}$. But this implies
\[\|U_{z_{\gamma}}^pKU_{z_{\gamma}}^pf_{z_{\gamma}}\|_p \leq \|U_{z_{\gamma}}^pKU_{z_{\gamma}}^pf_{\alpha}\|_p + \|U_{z_{\gamma}}^pKU_{z_{\gamma}}^p(f_{z_{\gamma}} - f_{\alpha})\|_p \leq \|K_{z_{\gamma}}f_{\alpha}\|_p + \norm{K}\|f_{z_{\gamma}} - f_{\alpha}\|_p,\]
which converges to $0$ by assumption. Thus $\lim\limits_{z_{\gamma} \to x} (\Bc(K))(z_{\gamma}) = 0$. As the net $(z_{\gamma})$ was arbitrary, we get $\lim\limits_{z \to \partial\Omega} (\Bc(K))(z) = 0$.

Conversely, let $x \in \beta\Omega \setminus \Omega$ and assume that $\lim\limits_{z \to \partial\Omega} (\Bc(K))(z) = 0$. Choose a net $(z_{\gamma})$ in $\Omega$ that converges to $x$. By Proposition \ref{prop6},
\[K_x = \slim\limits_{z_{\gamma} \to x} U_{z_{\gamma}}^pKU_{z_{\gamma}}^p|_{A^p_{\nu}}.\]
But let us consider
\[K_xT_{b_x}^{-1} = \slim\limits_{z_{\gamma} \to x} K_{z_{\gamma}}T_{b_{z_{\gamma}}}^{-1} = \slim\limits_{z_{\gamma} \to x} U_{z_{\gamma}}^pK(U_{z_{\gamma}}^q|_{A^q_{\nu}})^*\]
here instead (cf.~Remark \ref{rem_limit_operator_definition}, Proposition \ref{prop_T_b_x}). Of course, $K_x = 0$ if and only if $K_xT_{b_x}^{-1} = 0$, so it suffices to show $\slim\limits_{z_{\gamma} \to x} U_{z_{\gamma}}^pK(U_{z_{\gamma}}^q|_{A^q_{\nu}})^* = 0$. The reason why we want to consider this limit instead is the following computation:
\begin{align*}
\left((U_z^q|_{A^q_{\nu}})^*k_{\zeta}^{(p)}\right)(w) &= (T_{b_z}U_z^pk_{\zeta}^{(p)})(w)\\
&= \int_{\Omega} \frac{h(\zeta,\zeta)^{\frac{\nu+g}{q}}}{h(\phi_z(y),\zeta)^{\nu+g}} \frac{h(z,z)^{\frac{\nu+g}{p}}}{h(y,z)^{\frac{2(\nu+g)}{p}}} \frac{h(z,y)^{(\nu+g)(\frac{1}{q} - \frac{1}{p})}}{h(y,z)^{(\nu+g)(\frac{1}{q} - \frac{1}{p})}} h(w,y)^{-\nu-g} \, \mathrm{d}v_{\nu}(y)\\
&= h(\zeta,\zeta)^{\frac{\nu+g}{q}}h(z,z)^{\frac{\nu+g}{p}} \int_{\Omega} \frac{h(y,z)^{\nu+g}}{h(z,\zeta)^{\nu+g}h(y,\phi_z(\zeta))^{\nu+g}} \frac{h(z,y)^{(\nu+g)(\frac{1}{q} - \frac{1}{p})}}{h(y,z)^{\nu+g}}\\
&\qquad \qquad \qquad \qquad \qquad \qquad \cdot h(w,y)^{-\nu-g} \, \mathrm{d}v_{\nu}(y)\\
&= \frac{h(\zeta,\zeta)^{\frac{\nu+g}{q}}h(z,z)^{\frac{\nu+g}{p}}}{h(z,\zeta)^{\nu+g}} \overline{\int_{\Omega} \frac{h(y,z)^{(\nu+g)(\frac{1}{q} - \frac{1}{p})}}{h(y,w)^{\nu+g}} h(\phi_z(\zeta),y)^{-\nu-g} \, \mathrm{d}v_{\nu}(y)}\\
&= \frac{h(\zeta,\zeta)^{\frac{\nu+g}{q}}h(z,z)^{\frac{\nu+g}{p}}}{h(z,\zeta)^{\nu+g}} \frac{h(z,\phi_z(\zeta))^{(\nu+g)(\frac{1}{q} - \frac{1}{p})}}{h(w,\phi_z(\zeta))^{\nu+g}}\\
&= \frac{h(\zeta,\zeta)^{\frac{\nu+g}{q}}h(z,z)^{\frac{\nu+g}{q}}}{h(z,\zeta)^{\frac{2(\nu+g)}{q}}h(w,\phi_z(\zeta))^{\nu+g}}\\
&= \frac{h(\zeta,z)^{\frac{(\nu+g)}{q}}}{h(z,\zeta)^{\frac{(\nu+g)}{q}}} \frac{h(\phi_z(\zeta),\phi_z(\zeta))^{\frac{\nu+g}{q}}}{h(w,\phi_z(\zeta))^{\nu+g}}\\
&= \frac{h(\zeta,z)^{\frac{(\nu+g)}{q}}}{h(z,\zeta)^{\frac{(\nu+g)}{q}}} k_{\phi_z(\zeta)}^{(p)}(w)
\end{align*}
for all $w,z,\zeta \in \Omega$, where we used the usual transformation identities a few times and the fact that $P_{\nu}$ is the identity on holomorphic functions. Note that the overline indicates complex conjugation here. Similarly, we get $\left((U_z^p|_{A^p_{\nu}})^*k_{\zeta}^{(q)}\right)(w) = \frac{h(\zeta,z)^{\frac{(\nu+g)}{p}}}{h(z,\zeta)^{\frac{(\nu+g)}{p}}} k_{\phi_z(\zeta)}^{(q)}(w)$ for all $w,z,\zeta \in \Omega$. Thus
\begin{align} \label{Berezin_transform_shifted_operator}
(\Bc(K_{z_{\gamma}}T_{b_{z_{\gamma}}}^{-1}))(\zeta) &= \int_{\Omega} \left(U_{z_{\gamma}}^pK(U_{z_{\gamma}}^q|_{A^q_{\nu}})^*k_{\zeta}^{(p)}\right)(w)\overline{k_{\zeta}^{(q)}(w)} \, \mathrm{d}v_{\nu}(w)\notag\\
&= \frac{h(\zeta,z_{\gamma})^{\frac{(\nu+g)}{q}}}{h(z_{\gamma},\zeta)^{\frac{(\nu+g)}{q}}}\frac{h(z_{\gamma},\zeta)^{\frac{(\nu+g)}{p}}}{h(\zeta,z_{\gamma})^{\frac{(\nu+g)}{p}}} \int_{\Omega} (Kk_{\phi_{z_{\gamma}}(\zeta)}^{(p)})(w)\overline{k_{\phi_{z_{\gamma}}(\zeta)}^{(q)}(w)} \, \mathrm{d}v_{\nu}(w)\notag\\
&= b_{z_{\gamma}}(\zeta)^{-1}(\Bc(K))(\phi_{z_{\gamma}}(\zeta)).
\end{align}
As $\abs{b_{z_{\gamma}}(\zeta)^{-1}} = 1$, $\phi_{z_{\gamma}}(\zeta) \to \partial\Omega$ and $K_{z_{\gamma}}T_{b_{z_{\gamma}}}^{-1} \to K_xT_{b_x}^{-1}$ strongly as $z_{\gamma} \to x$, we get
\[(\Bc(K_xT_{b_x}^{-1}))(\zeta) = \lim\limits_{z_{\gamma} \to x} (\Bc(K_{z_{\gamma}}T_{b_{z_{\gamma}}}^{-1}))(\zeta) = \lim\limits_{z_{\gamma} \to x} b_{z_{\gamma}}(\zeta)^{-1}(\Bc(K))(\phi_{z_{\gamma}}(\zeta)) = 0\]
for all $\zeta \in \Omega$. Hence $K_xT_{b_x}^{-1} = 0$ and thus $K_x = 0$.
\end{proof}

For $p = 2$ things are a little bit simpler because $\Tf_{2,\nu}$ is an irreducible $C^*$-algebra containing a non-trivial compact operator and hence contains all compact operators. As Toeplitz operators are band-dominated (Proposition \ref{Projections_band_dominated}, see also Remark \ref{rem_Toeplitz}) the next corollary immediately follows.

\begin{cor} \label{cor2}
An operator $K \in \Lc(A^2_{\nu})$ is compact if and only if $K \in \Tf_{2,\nu}$ and $\lim\limits_{z \to \partial\Omega} (\Bc(K))(z) = 0$.
\end{cor}

\section{Fredholmness} \label{Fredholmness}

In the previous section we showed that compactness can be characterized in terms of limit operators. In this section we show that the same can be done with Fredholmness, i.e.~we show that a band-dominated operator is Fredholm if and only if all of its limit operators are invertible. As we gathered all the ingredients we need in the previous sections, we may follow now the lines of \cite{Hagger} to obtain the result. One direction is actually quite easy and follows directly from the compactness characterization:

\begin{prop} \label{Fredholmness_necessary_condition}
Let $A \in \Lc(A^p_{\nu})$ be band-dominated. If $A$ is Fredholm, then $A_x$ is invertible for every $x \in \beta\Omega \setminus \Omega$ and $\sup\limits_{x \in \beta\Omega \setminus \Omega} \norm{A_x^{-1}} < \infty$. Moreover, if $B$ is a Fredholm regularizer of $A$ and $(z_{\gamma})$ is a net in $\Omega$ that converges to $x \in \beta\Omega \setminus \Omega$, then $U_{z_{\gamma}}^pBU_{z_{\gamma}}^p$ converges strongly to $A_x^{-1}$ as $z_{\gamma} \to x$.
\end{prop}

\begin{proof}
Let $B$ be a Fredholm regularizer of $A$ and denote by $Q_{\nu} := I - P_{\nu} \in \BDO^p_{\nu}$ the complementary projection to $P_{\nu} \in \BDO^p_{\nu}$. Then $BP_{\nu} + Q_{\nu} \in \Lc(L^p_{\nu})$ is a Fredholm regularizer of $AP_{\nu} + Q_{\nu}$. By Proposition \ref{prop0}, this implies $BP_{\nu} + Q_{\nu} \in \BDO^p_{\nu}$ and hence $B$ is band-dominated. Similarly, $AB$ is band-dominated. Therefore, by Proposition \ref{prop6}, for every net $(z_{\gamma})$ converging to some $x \in \beta\Omega \setminus \Omega$ the strong limits of $U_{z_{\gamma}}^pBU_{z_{\gamma}}^p|_{A^p_{\nu}}$ and $U_{z_{\gamma}}^pABU_{z_{\gamma}}^p|_{A^p_{\nu}}$ exist and they are equal to $B_x$ and $(AB)_x$, respectively. Moreover, $AB-I$ and $BA-I$ are compact, hence  $(AB-I)_x = (BA-I)_x = 0$ by Theorem \ref{thm5}. It follows
\[0 = (AB-I)_x = A_xB_x - I\]
by Corollary \ref{limit_operator_properties2}, i.e.~$A_xB_x = I$. Together with the reversed equality we get $B_x = A_x^{-1}$ and hence $\norm{A_x^{-1}} \leq \norm{B}$ by Corollary \ref{limit_operator_properties2} again.
\end{proof}

The other direction is more difficult to show and needs some more preparation. For $p \leq 2$ and $\alpha = (\frac{2}{p} - 1)g + \frac{2\nu}{p}$ and $A \in \Lc(A^p_{\nu})$ we will use the notation $\hat{A}$ for the extension $AP_{\alpha} + Q_{\alpha} \in \Lc(L^p_{\nu})$, where $Q_{\alpha} := I - P_{\alpha}$ is the complementary projection. Note that since
\[AP_{\alpha} + Q_{\alpha} = AP_{\nu}P_{\alpha} + Q_{\alpha} \quad \text{and} \quad AP_{\nu} = (AP_{\alpha} + Q_{\alpha})P_{\nu},\]
$A$ is band-dominated if and only if $\hat{A}$ is (cf.~Proposition \ref{Projections_band_dominated}). Also note that
\[\widehat{A_x} = (AP_{\nu})_xP_{\alpha} + Q_{\alpha} = \wlim\limits_{z_{\gamma} \to x} U_{z_{\gamma}}^pAP_{\nu}U_{z_{\gamma}}^pP_{\alpha} + Q_{\alpha} = \wlim\limits_{z_{\gamma} \to x} U_{z_{\gamma}}^p(AP_{\alpha} + Q_{\alpha})U_{z_{\gamma}}^p = (\hat{A})_x\]
for any net $(z_{\gamma})$ coverging to $x \in \beta\Omega$ by Proposition \ref{interchange_prop}. We may therefore just write $\hat{A}_x$ without creating any ambiguities.

\begin{lem} \label{lem7}
Let $p \leq 2$, $\alpha = (\frac{2}{p} - 1)g + \frac{2\nu}{p}$, let $\xi \in L^{\infty}(\Omega)$ have compact support and let $A \in \Lc(A^p_{\nu})$ be band-dominated. Further assume that $(z_{\gamma})$ is a net in $\Omega$ converging to some $x \in \beta\Omega$ such that $A_x$ is invertible. Then there is a $\gamma_0$ such that for all $\gamma \geq \gamma_0$ there are operators $B_\gamma,C_{\gamma} \in \Lc(L^p_{\nu})$ with $\norm{B_{\gamma}}, \norm{C_{\gamma}} \leq 2\left(\norm{A_x^{-1}}\norm{P_{\alpha}} + \norm{Q_{\alpha}}\right)$ and
\[B_{\gamma}\hat{A}M_{\xi \circ \phi_{z_{\gamma}}} = M_{\xi \circ \phi_{z_{\gamma}}} = M_{\xi \circ \phi_{z_{\gamma}}}\hat{A}C_{\gamma}.\]
\end{lem}

\begin{proof}
First observe that $p \leq 2$ implies $\alpha \geq \nu$ and hence $(\alpha,\nu,p)$ is always admissible if $(\nu,\nu,p)$ is, which is assumed throughout this paper (see Section \ref{BSDs}, in particular Equation \eqref{admissible2}). Moreover, if $x \in \Omega$ and $A_x = U_x^pAU_x^p|_{A^p_{\nu}}$ is invertible, then $A$ is invertible and so the assertion holds trivially. We may therefore assume that $x \in \beta\Omega \setminus \Omega$.

Let $D(0,R) := \set{z \in \Omega : \beta(0,z) < R}$, where $R > 0$ is chosen sufficiently large such that $\supp\xi \subseteq D(0,R)$. By Proposition \ref{prop6}, $U_{z_{\gamma}}^pAU_{z_{\gamma}}^p|_{A^p_{\nu}}$ converges strongly to $A_x$. Moreover, the operator $P_{\alpha}M_{\chi_{D(0,R)}}$ is compact by \cite[Proposition 15]{Hagger}. Combining these facts and using Proposition \ref{interchange_prop}, we get
\begin{align*}
\norm{(\hat{A}_{z_{\gamma}} - \hat{A}_x)M_{\chi_{D(0,R)}}} &= \norm{\left(U_{z_{\gamma}}^p(AP_{\alpha} + Q_{\alpha})U_{z_{\gamma}}^p - A_xP_{\alpha} - Q_{\alpha})\right)M_{\chi_{D(0,R)}}}\\
&= \norm{\left(U_{z_{\gamma}}^pAU_{z_{\gamma}}^p - A_x\right)P_{\alpha}M_{\chi_{D(0,R)}}} \to 0
\end{align*}
as $z_{\gamma} \to x$. $\hat{A}_x$ is invertible with $(\hat{A}_x)^{-1} = \widehat{A_x^{-1}}$, which implies that there exists a $\gamma_0$ such that $R_{\gamma} := (\hat{A}_x)^{-1}(\hat{A}_{z_{\gamma}} - \hat{A}_x)M_{\chi_{D(0,R)}}$ satisfies $\norm{R_{\gamma}} < \frac{1}{2}$ for all $\gamma \geq \gamma_0$. In particular, $I + R_{\gamma} \in \Lc(L^p_{\nu})$ is invertible for all $\gamma \geq \gamma_0$. Using Proposition \ref{multiplication_shift}, it is now easy to see that
\[U_{z_{\gamma}}^p(I + R_{\gamma})^{-1}(\hat{A}_x)^{-1}U_{z_{\gamma}}^p\hat{A}M_{\xi \circ \phi_{z_{\gamma}}} = M_{\xi \circ \phi_{z_{\gamma}}}\]
and the first assertion follows (cf.~\cite[Proposition 19]{Hagger}).

For the second assertion note that $M_{\chi_{D(0,R)}}P_{\alpha}$ is compact as well (see \cite[Proposition 15]{Hagger}). Thus
\[\norm{M_{\chi_{D(0,R)}}(\hat{A}_{z_{\gamma}} - \hat{A}_x)} = \norm{M_{\chi_{D(0,R)}}P_{\alpha}\left(U_{z_{\gamma}}^pAU_{z_{\gamma}}^p - A_x\right)P_{\alpha}} \to 0\]
and we obtain
\[M_{\xi \circ \phi_{z_{\gamma}}}\hat{A}U_{z_{\gamma}}^p(\hat{A}_x)^{-1}(I + S_{\gamma})^{-1}U_{z_{\gamma}}^p = M_{\xi \circ \phi_{z_{\gamma}}}\]
for sufficiently large $\gamma$ and $S_{\gamma} := M_{\chi_{D(0,R)}}(\hat{A}_{z_{\gamma}} - \hat{A}_x)(\hat{A}_x)^{-1}$.
\end{proof}

Now we are ready to prove the other direction. Together with Proposition \ref{Fredholmness_necessary_condition}, we get the following theorem.

\begin{thm} \label{thm1}
Let $A \in \Lc(A^p_{\nu})$ be band-dominated. Then $A$ is Fredholm if and only if $A_x$ is invertible for every $x \in \beta\Omega \setminus \Omega$ and $\sup\limits_{x \in \beta\Omega \setminus \Omega} \norm{A_x^{-1}} < \infty$.
\end{thm}

\begin{proof}
In Proposition \ref{Fredholmness_necessary_condition} we have seen that if $A$ is Fredholm, then $A_x$ is invertible for every $x \in \beta\Omega \setminus \Omega$ and $\sup\limits_{x \in \beta\Omega \setminus \Omega} \norm{A_x^{-1}} < \infty$.

For the converse assume that $A$ is a band-dominated operator such that all limit operators $A_x$ are invertible and their inverses are uniformly bounded. Without loss of generality we may assume $p \leq 2$ because otherwise we could just pass to the adjoint and use Corollary \ref{cor_adjoints}. As in Lemma \ref{lem7} we set $\alpha = (\frac{2}{p} - 1)g + \frac{2\nu}{p}$.

Let $\psi_{j,t}$ be the functions defined in Section \ref{BSDs} and assume that $A$ is not Fredholm. Note that
\[[\hat{A},P_{\alpha}] = (AP_{\alpha} + Q_{\alpha})P_{\alpha} - P_{\alpha}(AP_{\alpha} + Q_{\alpha}) = AP_{\alpha} - P_{\alpha}AP_{\alpha} = 0\]
as $A \in \Lc(A^p_{\nu})$. Moreover, $\hat{A}$ is band-dominated because $A$ is. In this particular case \cite[Proposition 17]{Hagger} provides a criterion for $A = \hat{A}|_{A^p_{\nu}}$ to be Fredholm. As this would contradict our assumption, this criterion cannot be satisfied. In short, its negation reads
\[\nexists M > 0 : \forall \, t \in (0,1) \, \exists j_0 \in \N : \forall \, j \geq j_0 \, \exists B_{j,t},C_{j,t} \in \overline{B_{\norm{\cdot}}(0,M)} : B_{j,t}\hat{A}M_{\psi_{j,t}} = M_{\psi_{j,t}} = M_{\psi_{j,t}}\hat{A}C_{j,t},\]
where $\overline{B_{\norm{\cdot}}(0,M)} := \set{B \in \Lc(L^p_{\nu}) : \norm{B} \leq M}$ denotes the closed ball of radius $M$ in $\Lc(L^p_{\nu})$. This is equivalent to
\begin{align*}
\forall \, M > 0 \, \exists t \in (0,1) : \forall \, j_0 \in \N \, \exists j \geq j_0 : \forall B,C \in \, &\overline{B_{\norm{\cdot}}(0,M)} : B\hat{A}M_{\psi_{j,t}} \neq M_{\psi_{j,t}} \text{or } M_{\psi_{j,t}} \neq M_{\psi_{j,t}}\hat{A}C.
\end{align*}
In particular, there is a $t \in (0,1)$ and a strictly increasing sequence $(j_m)_{m \in \N}$ such that
\[B\hat{A}M_{\psi_{j_m,t}} \neq M_{\psi_{j_m,t}} \quad \text{or} \quad M_{\psi_{j_m,t}} \neq M_{\psi_{j_m,t}}\hat{A}C\]
for all $m \in \N$ and all $B \in \Lc(L^p_{\nu})$ with $\norm{B} \leq M := 2\left(\sup\limits_{x \in \beta\Omega \setminus \Omega} \norm{A_x^{-1}}\norm{P_{\alpha}} + \norm{Q_{\alpha}}\right)$. By choosing a suitable subsequence if necessary, we may assume that either always the first or always the second inequality happens. As both cases can be treated in the same way, we may assume that
\begin{equation} \label{thm1_eq}
B\hat{A}M_{\psi_{j_m,t}} \neq M_{\psi_{j_m,t}}
\end{equation}
for all $m \in \N$ and $B \in \overline{B_{\norm{\cdot}}(0,M)}$. Now by property $(e)$ of the functions $\psi_{j,t}$, the diameters $\diam_{\beta} \supp \psi_{j,t}$ are bounded by a constant not depending on $j \in \N$. Thus every $\supp \psi_{j_m,t}$ is contained in a Bergman ball $D(w_m,R)$ for a fixed radius $R$. As every $\supp \psi_{j_m,t}$ at least contains a Bergman ball of radius $\frac{1}{t}$ and for every $z \in \Omega$ the set $\set{m \in \N : z \in \supp \psi_{j_m,t}}$ has at most $N$ elements, it is also clear that the sequence of midpoints $(w_m)_{m \in \N}$ tends to the boundary $\partial\Omega$ as $m \to \infty$. By compactness of $\beta\Omega$, this sequence has a subnet $(w_{m_\gamma})$ that converges to some $x \in \beta\Omega \setminus \Omega$. Thus, by choosing $\xi = \chi_{D(0,R)}$ in Lemma \ref{lem7}, we obtain a $\gamma_0$ such that for all $\gamma \geq \gamma_0$ there is an operator $B_\gamma \in \overline{B_{\norm{\cdot}}(0,M)}$ with
\[B_\gamma\hat{A}M_{\chi_{D(w_{m_\gamma},R)}} = B_\gamma\hat{A}M_{\chi_{D(0,R)} \circ \phi_{w_{m_{\gamma}}}} = M_{\chi_{D(0,R)} \circ \phi_{w_{m_{\gamma}}}} = M_{\chi_{D(w_{m_\gamma},R)}}.\]
By multiplying with $M_{\psi_{j_{m_\gamma,t}}}$ from the right, we obtain a contradiction to \eqref{thm1_eq}. Therefore $A$ has to be Fredholm.
\end{proof}

Next we will show that the uniform boundedness condition for the inverses is actually redundant. The argument is very similar to the unit ball and the sequence space case, cf.~\cite{Hagger,LiSe}. First, we need an analogue of $\vertiii{\cdot|_F}$ (cf.~Section \ref{compactness}).

Recall that $r_t = \sup\limits_{j \in \N} \diam_{\beta} \supp \varphi_{j,t} < \infty$. For every $t \in (0,1)$, $A \in \Lc(L^p_{\nu})$ and every Borel set $F \subseteq \Omega$ we define
\[\nu_t(A|_F) := \inf\set{\norm{Af} : f \in L^p_{\nu}, \norm{f} = 1, \supp f \subseteq D(w,r_t) \cap F \text{ for some } w \in \Omega}\]
and
\[\nu(A|_F) := \inf\set{\norm{Af} : f \in L^p_{\nu}, \norm{f} = 1, \supp f \subseteq F}.\]
Moreover, $\nu(A) := \nu(A|_{\Omega})$.

The following lemma is immediate (see e.g.~\cite[Proposition 22]{Hagger} or \cite[Lemma 2.38]{Lindner}).

\begin{lem} \label{lem8}
For all $A,B \in \Lc(L^p_{\nu})$, $t \in (0,1)$ and all Borel sets $F \subseteq \Omega$ it holds 
\[\abs{\nu(A|_F) - \nu(B|_F)} \leq \norm{(A - B)M_{\chi_F}} \quad \text{and} \quad \abs{\nu_t(A|_F) - \nu_t(B|_F)} \leq \norm{(A - B)M_{\chi_F}}.\]
\end{lem}

The next proposition is the analogue of Proposition \ref{prop9}. A large part of the proof is actually the same, so we just sketch it here.

\begin{prop} \label{prop10}
Let $A \in \BDO^p_{\nu}$. Then for every $\epsilon > 0$ there exists a $t \in (0,1)$ such that for every Borel set $F \subseteq \Omega$ and every $B \in \set{A} \cup \set{A_x : x \in \beta\Omega \setminus \Omega}$ it holds
\[\nu(B|_F) \leq \nu_t(B|_F) \leq \nu(B|_F) + \epsilon.\]
\end{prop}

\begin{proof}
By Proposition \ref{limit_operator_properties} and Lemma \ref{lem8}, we may assume that $A$ is a band operator. Moreover, the first inequality is clear by definition. For the second inequality observe that all limit operators $A_x$ have the same band width as $A$. So let $B \in \{A\} \cup \set{A_x : x \in \beta\Omega \setminus \Omega}$, $F \subseteq \Omega$ a Borel set and choose $f \in L^p_{\nu}$ with $\norm{f} = 1$ and $\supp f \subseteq F$ such that
\[\norm{Bf} \leq \nu(B|_F) + \frac{\epsilon}{2}.\]
Moreover, let $\varphi_{j,t}$ and $\psi_{j,t}$ be defined as in Section \ref{BSDs}. Then
\begin{align*}
\left(\sum\limits_{j = 1}^{\infty} \norm{BM_{\varphi_{j,t}^{1/p}}f}^p\right)^{1/p} &= \left(\sum\limits_{j = 1}^{\infty} \norm{BM_{\varphi_{j,t}^{1/p}}M_{\psi_{j,t}}f}^p\right)^{1/p}\\
&\leq \left(\sum\limits_{j = 1}^{\infty} \norm{M_{\varphi_{j,t}^{1/p}}Bf}^p\right)^{1/p} + \left(\sum\limits_{j = 1}^{\infty} \norm{M_{\varphi_{j,t}^{1/p}}BM_{1-\psi_{j,t}}f}^p\right)^{1/p}\\
&\qquad + \left(\sum\limits_{j = 1}^{\infty} \norm{[B,M_{\varphi_{j,t}^{1/p}}]M_{\psi_{j,t}}f}^p\right)^{1/p}
\end{align*}
by Minkowski's inequality. As in the proof of Proposition \ref{prop9}, the first term is equal to $\norm{Bf}$ and the other two terms tend to $0$ (uniformly in $f$ and $B$) as $t \to 0$. Thus, for sufficiently small $t$, we get
\[\left(\sum\limits_{j = 1}^{\infty} \norm{BM_{\varphi_{j,t}^{1/p}}f}^p\right)^{1/p} \leq  \norm{Bf} + \frac{\epsilon}{2} \leq \nu(B|_F) + \epsilon = \left(\nu(B|_F) + \epsilon\right)\left(\sum\limits_{j = 1}^{\infty} \norm{M_{\varphi_{j,t}^{1/p}}f}^p\right)^{1/p}.\]
This implies, in particular, that there exists a $j \in \N$ such that
\[\norm{BM_{\varphi_{j,t}^{1/p}}f} \leq \left(\nu(B|_F) + \epsilon\right)\norm{M_{\varphi_{j,t}^{1/p}}f}\]
for sufficiently small $t$. As $\supp \left(M_{\varphi_{j,t}^{1/p}}f\right) \subseteq \supp \varphi_{j,t} \subseteq D(w,r_t)$ for some $w \in \Omega$, the assertion follows.
\end{proof}

The next lemma allows us to centralize certain functions. This will be crucial in the subsequent lemma. Again, we focus on the case $p \leq 2$ as we can always pass to the adjoint if necessary. Recall that we defined $\hat{A} := AP_{\alpha} + Q_{\alpha}$ for $\alpha = (\frac{2}{p} - 1)g + \frac{2\nu}{p}$.

\begin{lem} \label{lem5}
Let $p \leq 2$, $\alpha = (\frac{2}{p} - 1)g + \frac{2\nu}{p}$ and $f \in L^p_{\nu}$ with $\supp f \subseteq D(w,r)$ for some $w \in \Omega$ and $r > 0$. If $A \in \Lc(A^p_{\nu})$ is band-dominated, then for every $x \in \beta\Omega \setminus \Omega$ there exist $y \in \beta\Omega \setminus \Omega$ and $g \in L^p_{\nu}$ with $\supp g \subseteq D(0,r)$ and $\norm{g} = \norm{f}$ such that $\|\hat{A}_xf\| = \|\hat{A}_yg\|$. Moreover,
\[\nu(\hat{A}_y|_{D(0,r + \beta(0,w))}) \leq \nu(\hat{A}_x|_{D(0,r)}).\]
\end{lem}

\begin{proof}
Another direct computation yields
\[(U_z^pU_w^pf)(\zeta) = (f \circ \phi_w \circ \phi_z)(\zeta) \frac{h(\phi_z(w),\phi_z(w))^{\frac{\nu+g}{p}}}{h(\zeta,\phi_z(w))^{\frac{2(\nu+g)}{p}}} \left(\frac{h(w,z)}{\abs{h(w,z)}}\right)^{\frac{2(\nu+g)}{p}}\]
for $w,z \in \Omega$ (see e.g.~\cite[Lemma 2.8]{NaZheZho}). As $(\phi_w \circ \phi_z \circ \phi_{\phi_z(w)})(0) = 0$ and $\phi_{\phi_z(w)}$ is an involution, we obtain $\phi_w \circ \phi_z = V \circ \phi_{\phi_z(w)}$ for some $V \in K$ by Cartan's linearity theorem (see Section \ref{BSDs}). This implies
\[U_z^pU_w^p = \left(\frac{h(w,z)}{\abs{h(w,z)}}\right)^{\frac{2(\nu+g)}{p}}U_{\phi_z(w)}^pV_*,\]
where $V_*f := f \circ V$ is a composition operator and, by taking inverses, also
\[U_w^pU_z^p = \left(\frac{h(w,z)}{\abs{h(w,z)}}\right)^{-\frac{2(\nu+g)}{p}}V_*^{-1}U_{\phi_z(w)}^p.\]
Combining these two equalities, we get
\[U_w^pU_z^pAU_z^pU_w^p|_{A^p_{\nu}} = V_*^{-1}U_{\phi_z(w)}^pAU_{\phi_z(w)}^pV_*|_{A^p_{\nu}}\]
for any band-dominated $A \in \Lc(A^p_{\nu})$. Note that $V$ of course depends on $w$ and $z$. We now fix $w \in \Omega$ and choose a net $(z_{\gamma})$ that converges to some $x \in \beta\Omega \setminus \Omega$. As $K$ is a closed subgroup of the unitary group, there is a subnet of $(z_{\gamma})$, again denoted by $(z_{\gamma})$, such that $V$ converges to some $\tilde{V} \in K$ as $z_{\gamma} \to x$. In particular, $V_*|_{A^p_{\nu}}$ converges strongly to $\tilde{V}_*|_{A^p_{\nu}}$ and $V_*^{-1}|_{A^p_{\nu}}$ converges strongly to $\tilde{V}^{-1}_*|_{A^p_{\nu}}$. Moreover, using Corollary \ref{cor1}, we may assume that $U_{\phi_{z_{\gamma}}(w)}^pAU_{\phi_{z_{\gamma}}(w)}^p|_{A^p_{\nu}}$ converges strongly to $A_y$ for some $y \in \beta\Omega$. Since $\phi_{z_{\gamma}}(w) \to \partial\Omega$ as $z_{\gamma} \to \partial\Omega$, it is clear that $y \in \beta\Omega \setminus \Omega$. As the limit of a strongly convergent net is unique and $\slim\limits_{z_{\gamma} \to x} U_{z_{\gamma}}^pAU_{z_{\gamma}}^p|_{A^p_{\nu}} = A_x$ by Proposition \ref{prop6}, we obtain
\[U_w^pA_xU_w^p|_{A^p_{\nu}} = \tilde{V}_*^{-1}A_y\tilde{V}_*|_{A^p_{\nu}}.\]
Now observe that $P_{\alpha}$ commutes with both $U_w^p$ and $\tilde{V}_*$. Indeed, the former was shown in Proposition \ref{interchange_prop}, the latter follows from the fact that $h(z,w)$ is invariant under $K$ (see Section \ref{BSDs}):
\begin{align*}
(P_{\alpha}\tilde{V}_*f)(z) &= \int_{\Omega} f(\tilde{V}w)h(z,w)^{-\alpha-g} \, \mathrm{d}v_{\alpha}(w)\\
&= \int_{\Omega} f(w)h(z,\tilde{V}^{-1}w)^{-\alpha-g} \, \mathrm{d}v_{\alpha}(w)\\
&= \int_{\Omega} f(w)h(\tilde{V}z,w)^{-\alpha-g} \, \mathrm{d}v_{\alpha}(w)\\
&= (\tilde{V}_*P_{\alpha}f)(z).
\end{align*}
Therefore we also have
\[U_w^p\hat{A}_xU_w^p = \tilde{V}_*^{-1}\hat{A}_y\tilde{V}_*.\]

Now clearly, if $f \in L^p_{\nu}$ with $\supp f \subseteq D(w,r)$ for some $w \in \Omega$ and $r > 0$, then $g := \tilde{V}_*U_w^pf$ satisfies $\norm{g} = \norm{f}$ and $\|\hat{A}_xf\| = \|\hat{A}_yg\|$. Moreover,
\[\supp g = \overline{\set{z \in \Omega : g(z) \neq 0}} = \overline{\set{z \in \Omega : f(\phi_w(\tilde{V}z)) \neq 0}} \subseteq D(0,r)\]
because $\phi_w(D(0,r)) = D(w,r)$ and $\tilde{V}(D(0,r)) = D(0,r)$.

For the second assertion consider $f \in L^p_{\nu}$ with $\supp f \subseteq D(0,r)$. Then $g := \tilde{V}_*U_w^pf$ satisfies $\norm{g} = \norm{f}$, $\|\hat{A}_xf\| = \|\hat{A}_yg\|$ and $\supp g \subseteq \tilde{V}^{-1}(\phi_w^{-1}(D(0,r))) \subseteq D(0,r + \beta(0,w))$ as above.
\end{proof}

To actually show that the uniform boundedness condition is redundant, we show that the infimum $\inf\set{\nu(\hat{A}_x) : x \in \beta\Omega \setminus \Omega}$ is always attained. The assertion then follows from the fact that $\nu(B) = \norm{B^{-1}}^{-1}$ for invertible operators $B$.

\begin{lem} \label{lem6}
Let $p \leq 2$, $\alpha = (\frac{2}{p} - 1)g + \frac{2\nu}{p}$ and let $A \in \Lc(A^p_{\nu})$ be band-dominated. Then there exists a $y \in \beta\Omega \setminus \Omega$ with
\[\nu(\hat{A}_y) = \inf\set{\nu(\hat{A}_x) : x \in \beta\Omega \setminus \Omega}.\]
\end{lem}

\begin{proof}
Recall $r_t = \sup\limits_{j \in \N} \diam_{\beta}\supp \varphi_{j,t}$. Choosing $\epsilon = 2^{-(k+1)}$ for $k \in \N_0$ in Proposition \ref{prop10}, we obtain a sequence $(t_k)_{k \in \N_0}$ with $\nu_{t_k}(B|_F) \leq \nu(B|_F) + 2^{-(k+1)}$ for all $F \subseteq \Omega$ and $B \in \{\hat{A}\} \cup \set{\hat{A}_x : x \in \beta\Omega \setminus \Omega}$. Without loss of generality we may assume $r_{t_{k+1}} > 2r_{t_k}$ for all $k \in \N_0$. Now choose a sequence $(x_n)_{n \in \N}$ in $\beta\Omega \setminus \Omega$ such that
\[\lim\limits_{n \to \infty} \nu(\hat{A}_{x_n}) = \inf\set{\nu(\hat{A}_x) : x \in \beta\Omega \setminus \Omega}.\]
For every $n \in \N$ we may choose a function $f_n^0 \in L^p_{\nu}$ with $\|f_n^0\| = 1$, $\supp f_n^0$ contained in some $D(w,r_{t_n})$ and
\[\norm{\hat{A}_{x_n}f_n^0} \leq \nu_{t_n}(\hat{A}_{x_n}) + 2^{-(n+1)} \leq \nu(\hat{A}_{x_n}) + 2^{-n}.\]
Using Lemma \ref{lem5}, we obtain a boundary point $y_n^0 \in \beta\Omega \setminus \Omega$ and a function $g_n^0 \in L^p_{\nu}$ with $\|g_n^0\| = 1$, $\supp g_n^0 \subseteq D(0,r_{t_n})$ and
\[\norm{\hat{A}_{y_n^0}g_n^0} = \norm{\hat{A}_{x_n}f_n^0} \leq \nu(\hat{A}_{x_n}) + 2^{-n}.\]
Applying the same argument to $\hat{A}_{y_n^0}|_{D(0,r_{t_n})}$ we obtain a function $f_n^1 \in L^p_{\nu}$ with $\|f_n^1\| = 1$, $\supp f_n^1 \subseteq D(w,r_{t_{n-1}}) \cap D(0,r_{t_n})$ for some $w \in \Omega$ and
\[\norm{\hat{A}_{y_n^0}f_n^1} \leq \nu_{t_{n-1}}(\hat{A}_{y_n^0}|_{D(0,r_{t_n})}) + 2^{-n} \leq \nu(\hat{A}_{y_n^0}|_{D(0,r_{t_n})}) + 2^{-n+1}.\]
It is clear that $w \in D(0,r_{t_n}+r_{t_{n-1}})$, otherwise $D(w,r_{t_{n-1}}) \cap D(0,r_{t_n})$ would be empty. Using Lemma \ref{lem5} again, we can choose a boundary point $y_n^1 \in \beta\Omega \setminus \Omega$ and a function $g_n^1 \in L^p_{\nu}$ with $\|g_n^1\| = 1$, $\supp g_n^1 \subseteq D(0,r_{t_{n-1}})$ and
\[\norm{\hat{A}_{y_n^1}g_n^1} = \norm{\hat{A}_{y_n^0}f_n^1} \leq \nu(\hat{A}_{y_n^0}|_{D(0,r_{t_n})}) + 2^{-n+1}.\]
In particular,
\[\nu(\hat{A}_{y_n^1}|_{D(0,r_{t_{n-1}})}) \leq \nu(\hat{A}_{y_n^0}|_{D(0,r_{t_n})}) + 2^{-n+1} \leq \nu(\hat{A}_{x_n}) + 2^{-n+1} + 2^{-n}.\]

If we iterate this procedure for $k = 2, \ldots, n$, we obtain a boundary point $y_n^k \in \beta\Omega \setminus \Omega$, a midpoint $w \in D(0,r_{t_{n-k+1}} + r_{t_{n-k}})$, a function $f_n^k \in L^p_{\nu}$ with $\norm{f_n^k} = 1$, $\supp f_n^k \subseteq D(w,r_{t_{n-k}}) \cap D(0,r_{t_{n-k+1}})$ and
\[\norm{\hat{A}_{y_n^{k-1}}f_n^k} \leq \nu_{t_{n-k}}(\hat{A}_{y_n^{k-1}}|_{D(0,r_{t_{n-k+1}})}) + 2^{-n+k-1} \leq \nu(\hat{A}_{y_n^{k-1}}|_{D(0,r_{t_{n-k+1}})}) + 2^{-n+k}\]
and a function $g_n^k \in L^p_{\nu}$ with $\|g_n^k\| = 1$, $\supp g_n^k \subseteq D(0,r_{t_{n-k}})$ and
\[\norm{\hat{A}_{y_n^k}g_n^k} = \norm{\hat{A}_{y_n^{k-1}}f_n^k} \leq \nu(\hat{A}_{y_n^{k-1}}|_{D(0,r_{t_{n-k+1}})}) + 2^{-n+k}.\]
Thus, if we combine these estimates, we get
\begin{align*}
\nu(\hat{A}_{y_n^k}|_{D(0,r_{t_{n-k}})}) &\leq \nu(\hat{A}_{y_n^{k-1}}|_{D(0,r_{t_{n-k+1}})}) + 2^{-n+k} \leq \ldots \leq \nu(\hat{A}_{y_n^0}|_{D(0,r_{t_n})}) + 2^{-n+k} + \ldots + 2^{-n+1}\\
&\leq \nu(\hat{A}_{x_n}) + 2^{-n+k} + \ldots + 2^{-n+1} + 2^{-n} \leq \nu(\hat{A}_{x_n}) + 2^{-n+k+1}.
\end{align*}
Fix an integer $l \leq n$ and choose $k = n-l$. Then, by collecting all the shifts being made during the process above and repeatedly applying the second part of Lemma \ref{lem5}, we get
\begin{align*}
\nu(\hat{A}_{y_n^{n-l}}|_{D(0,r_{t_l})}) &\geq \nu(\hat{A}_{y_n^{n-l+1}}|_{D(0,r_{t_l} + r_{t_l} + r_{t_{l-1}})}) = \nu(\hat{A}_{y_n^{n-l+1}}|_{D(0,2r_{t_l} + r_{t_{l-1}})})\\
&\geq \nu(\hat{A}_{y_n^{n-l+2}}|_{D(0,2r_{t_l} + r_{t_{l-1}} + r_{t_{l-1}} + r_{t_{l-2}})}) = \nu(\hat{A}_{y_n^{n-l+2}}|_{D(0,2r_{t_l} + 2r_{t_{l-1}} + r_{t_{l-2}})})\\
&\geq \ldots \geq \nu(\hat{A}_{y_n^n}|_{D(0,2r_{t_l} + 2r_{t_{l-1}} + 2r_{t_{l-2}} + \ldots + 2r_{t_1} + r_{t_0})}) \geq \nu(\hat{A}_{y_n^n}|_{D(0,4r_{t_l})}),
\end{align*}
where we used $r_{t_{k+1}} > 2r_{t_k}$ for the last inequality.

Consider the diagonal sequence defined by $y_n := y_n^n$. Corollary \ref{cor1} implies that the sequence $(A_{y_n})_{n \in \N}$ has a strongly convergent subnet that converges to $A_y$ for some $y \in \beta\Omega \setminus \Omega$. Let us denote this subnet by $(A_{y_{n_{\gamma}}})$. Then
\[\norm{(\hat{A}_{y_{n_{\gamma}}} - \hat{A}_y)M_{\chi_{D(0,4r_{t_l})}}} = \norm{(A_{y_{n_{\gamma}}} - A_y)P_{\alpha}M_{\chi_{D(0,4r_{t_l})}}} \to 0\]
because $P_{\alpha}M_{\chi_{D(0,4r_{t_l})}} \in \Lc(L^p_{\nu})$ is a compact operator (see \cite[Proposition 15]{Hagger}). By Lemma \ref{lem8} we thus obtain $\nu(\hat{A}_{y_{n_{\gamma}}}|_{D(0,4r_{t_l})}) \to \nu(\hat{A}_y|_{D(0,4r_{t_l})})$. It follows
\begin{align*}
\nu(\hat{A}_y) &\leq \nu(\hat{A}_y|_{D(0,4r_{t_l})}) = \lim\limits_{\gamma} \nu(\hat{A}_{y_{n_{\gamma}}}|_{D(0,4r_{t_l})}) \leq \lim\limits_{\gamma} \nu(\hat{A}_{y_{n_{\gamma}}^{n_{\gamma}-l}}|_{D(0,r_{t_l})}) \leq \lim\limits_{\gamma} \nu(\hat{A}_{x_{n_{\gamma}}}) + 2^{-l+1}\\
&= \inf\set{\nu(\hat{A}_x) : x \in \beta\Omega \setminus \Omega} + 2^{-l+1}.
\end{align*}
Since $y$ does not depend on $l$, we get $\nu(\hat{A}_y) = \inf\set{\nu(\hat{A}_x) : x \in \beta\Omega \setminus \Omega}$ as claimed.
\end{proof}

Let us now summarize this section.

\renewcommand*{\themthm}{B}
\begin{mthm} \label{thm4}
Let $A \in \Lc(A^p_{\nu})$ be band-dominated. Then the following are equivalent:
\begin{itemize}
\item[$(i)$] $A$ is Fredholm,
\item[$(ii)$] $A_x$ is invertible and $\norm{A_x^{-1}} \leq \norm{(A + \Kc(A^p_{\nu}))^{-1}}$ for all $x \in \beta\Omega \setminus \Omega$,
\item[$(iii)$] $A_x$ is invertible for all $x \in \beta\Omega \setminus \Omega$ and $\sup\limits_{x \in \beta\Omega \setminus \Omega} \norm{A_x^{-1}} < \infty$,
\item[$(iv)$] $A_x$ is invertible for all $x \in \beta\Omega \setminus \Omega$,
\end{itemize}
\end{mthm}

\begin{proof}
The equivalence of $(i)$ and $(iii)$ is Theorem \ref{thm1}. Moreover, if $B$ is a Fredholm regularizer of $A$, Corollary \ref{limit_operator_properties2} and Proposition \ref{Fredholmness_necessary_condition} imply that the inverses are bounded by $\norm{B}$. As this is true for every regularizer $B$, this means $\norm{A_x^{-1}} \leq \norm{(A + \Kc(A^p_{\nu}))^{-1}}$ for all $x \in \beta\Omega \setminus \Omega$. Hence $(i)$ is also equivalent to $(ii)$.

Clearly, $(iii)$ implies $(iv)$. To show that (iv) implies (iii) it suffices to consider the case $p \leq 2$ (cf.~Proposition \ref{prop_T_b_x}). If $A_x$ is invertible, then $(\hat{A}_x)^{-1}$ is also invertible with $(\hat{A}_x)^{-1} = A_x^{-1}P_{\alpha} + Q_{\alpha}$. Moreover, it holds $\nu(B) = \norm{B^{-1}}^{-1} > 0$ for invertible operators $B$ (see e.g.~\cite[Lemma 2.35]{Lindner} for a quick proof). Therefore we get $\sup\limits_{x \in \beta\Omega \setminus \Omega} \|\hat{A}_x^{-1}\| < \infty$ by Lemma \ref{lem6}. Since $A_x^{-1} = \hat{A}_x^{-1}|_{A^p_{\nu}}$, this implies $(iii)$.
\end{proof}

The following corollary is now immediate.

\begin{cor} \label{cor3}
Let $A \in \Lc(A^p_{\nu})$ be band-dominated. Then
\[\spec_{\ess}(A) = \bigcup\limits_{x \in \beta\Omega \setminus \Omega} \spec(A_x).\]
\end{cor}

\section{Applications to Toeplitz operators} \label{applications}

As Toeplitz operators are band-dominated by Proposition \ref{Projections_band_dominated}, we obtain the following important corollary of Theorem \ref{thm4}.

\begin{cor} \label{thm4_Toeplitz}
Let $A \in \Tf_{p,\nu}$. Then $A$ is Fredholm if and only if all of its limit operators are invertible. In particular,
\[\spec_{\ess}(A) = \bigcup\limits_{x \in \beta\Omega \setminus \Omega} \spec(A_x).\]
\end{cor}

In the rest of this section we study some applications of Corollary \ref{thm4_Toeplitz} for particular Toeplitz operators. First, we observe the following:

\begin{prop} \label{prop_shifted_Toeplitz}
For $f \in L^{\infty}(\Omega)$ and $z \in \Omega$ we have $U_z^pT_fU_z^p = T_{b_z}^{-1}T_{(f \circ \phi_z)b_z}$.
\end{prop}

\begin{proof}
With the usual tricks we get
\begin{align*}
(U_z^pT_fU_z^p g)(w) &= \frac{h(z,z)^{\frac{\nu+g}{p}}}{h(w,z)^{\frac{2(\nu+g)}{p}}} \int_{\Omega} g(\phi_z(x)) \frac{h(z,z)^{\frac{\nu+g}{p}}}{h(x,z)^{\frac{2(\nu+g)}{p}}} f(x)h(\phi_z(w),x)^{-\nu-g} \, \mathrm{d}v_{\nu}(x)\\
&= h(w,z)^{(\frac{1}{q} - \frac{1}{p})(\nu+g)} \int_{\Omega} g(y)f(\phi_z(y)) h(y,z)^{(\frac{1}{p} - \frac{1}{q})(\nu+g)} h(w,y)^{-\nu-g} \, \mathrm{d}v_{\nu}(y)
\end{align*}
and hence
\begin{align*}
(T_{b_z}U_z^pT_fU_z^p g)(x) &= \int_{\Omega}\int_{\Omega} g(y)f(\phi_z(y)) \frac{h(z,w)^{(\frac{1}{q} - \frac{1}{p})(\nu+g)}}{h(y,z)^{(\frac{1}{q} - \frac{1}{p})(\nu+g)}} h(w,y)^{-\nu-g} h(x,w)^{-\nu-g} \, \mathrm{d}v_{\nu}(y) \, \mathrm{d}v_{\nu}(w)\\
&= \int_{\Omega} g(y)f(\phi_z(y)) \frac{h(z,y)^{(\frac{1}{q} - \frac{1}{p})(\nu+g)}}{h(y,z)^{(\frac{1}{q} - \frac{1}{p})(\nu+g)}} h(x,y)^{-\nu-g} \, \mathrm{d}v_{\nu}(y)\\
&= (T_{(f \circ \phi_z)b_z}g)(x)
\end{align*}
for all $g \in A^p_{\nu}$ and $x,z \in \Omega$.
\end{proof}

Combining Proposition \ref{prop_shifted_Toeplitz} with Proposition \ref{prop_T_b_x}, we get the following corollary.

\begin{cor} \label{cor_Toeplitz_limits}
Let $f \in L^{\infty}(\Omega)$, $x \in \beta\Omega \setminus \Omega$ and $(z_{\gamma})$ a net in $\Omega$ that converges to $x$. Further assume that $f \circ \phi_{z_{\gamma}}$ converges to a function $g \in L^{\infty}(\Omega)$ uniformly on compact subsets of $\Omega$. Then $(T_f)_x = T_{b_x}^{-1}T_{g \cdot b_x}$.
\end{cor}

For $p = 2$ this simplifies to $(T_f)_x = T_g$. Moreover, if $f$ is uniformly continuous with respect to the Bergman metric $\beta$, the condition in Corollary \ref{cor_Toeplitz_limits} is always satisfied. The set of bounded and uniformly continuous functions $f \from \Omega \to \C$ will be denoted by $\BUC(\Omega)$.

\begin{prop} \label{prop_BUC_limits}
Let $f \in \BUC(\Omega)$ and $x \in \beta\Omega \setminus \Omega$. Then $(T_f)_x = T_{b_x}^{-1}T_{g \cdot b_x}$, where $g \in \BUC(\Omega)$ is the pointwise limit of the net $(f \circ \phi_{z_{\gamma}})$.
\end{prop}

\begin{proof}
As $z \mapsto \phi_z(w)$ is a continuous function by Lemma \ref{lem2}, $z \mapsto f(\phi_z(w))$ is a bounded and continuous function for every $w \in \Omega$. Therefore there is a unique extension to the Stone-\v{C}ech compactification $\beta\Omega$. In particular, for every convergent net $(z_{\gamma})$ in $\Omega$ the net $f \circ \phi_{z_{\gamma}}$ converges pointwise to a bounded function $g$. As
\[\abs{g(w)-g(y)} = \lim\limits_{\gamma} \abs{f(\phi_{z_{\gamma}}(w)) - f(\phi_{z_{\gamma}}(y))}\]
and $\beta(\phi_{z_{\gamma}}(w),\phi_{z_{\gamma}}(y)) = \beta(w,y)$, $g$ is uniformly continuous just like $f$. It remains to show that $f \circ \phi_{z_{\gamma}} \to g$ uniformly on compact sets. So let $K \subset \Omega$ be compact and $\epsilon > 0$. As $f$ and $g$ are uniformly continuous, there is a $\delta > 0$ such that $\abs{f(w) - f(y)} < \epsilon$ and $\abs{g(w) - g(y)} < \epsilon$ whenever $\beta(w,y) < \delta$. Moreover, there is a finite set $\set{w_1, \ldots, w_k}$ such that $\min\limits_{j = 1, \ldots, k} \beta(w,w_j) < \delta$ for all $w \in K$. Now choose $\gamma$ sufficiently large such that $\abs{f(\phi_{z_{\gamma}}(w_j)) - g(w_j)} < \epsilon$ for all $j = 1, \ldots, k$. It follows
\[\abs{f(\phi_{z_{\gamma}}(w)) - g(w)} \leq \abs{f(\phi_{z_{\gamma}}(w)) - f(\phi_{z_{\gamma}}(w_j))} + \abs{f(\phi_{z_{\gamma}}(w_j)) - g(w_j)} + \abs{g(w_j)-g(w)} < 3\epsilon\]
for all $w \in K$, where $j$ is obviously chosen in such a way that $\beta(\phi_{z_{\gamma}}(w),\phi_{z_{\gamma}}(w_j)) = \beta(w,w_j) < \delta$, respectively. As $\epsilon$ was arbitrary, the conclusion follows.
\end{proof}

Proposition \ref{prop_BUC_limits} gets particularly simple if $g$ happens to be a constant function. As we will show next, this is the case if and only if the Berezin transform $\Bc(T_f)$ is of vanishing oscillation at the boundary $\partial\Omega$. For a bounded continuous function $f$, we define its oscillation at a point $z \in \Omega$ as
\[\Osc_z(f) := \sup\set{\abs{f(z) - f(w)} : w \in \Omega, \beta(z,w) \leq 1}.\]
We say that $f$ is of vanishing oscillation at the boundary, $f \in \VO_{\partial}(\Omega)$, if $\Osc_z(f) \to 0$ as $z \to \partial\Omega$.

\begin{thm} \label{thm7}
Let $A \in \Lc(A^p_{\nu})$ be band-dominated. Then $A_x$ is a multiple of the identity for every $x \in \beta\Omega \setminus \Omega$ if and only if $\Bc(A) \in \VO_{\partial}(\Omega)$. In this case $A_x = (\Bc(A))(x) \cdot I$.
\end{thm}

\begin{proof}
Assume that $\Bc(A)$ has vanishing oscillation at the boundary and let $(z_{\gamma})$ be a net in $\Omega$ that converges to $x \in \beta\Omega \setminus \Omega$. By Equation \eqref{Berezin_transform_shifted_operator}, which does not require compactness of $A$, we have
\begin{align*}
(\Bc(A_xT_{b_x}^{-1}))(\zeta) &= \lim\limits_{z_{\gamma} \to x} (\Bc(A_{z_{\gamma}}T_{b_{z_{\gamma}}}^{-1}))(\zeta) = \lim\limits_{z_{\gamma} \to x} b_{z_{\gamma}}(\zeta)^{-1}(\Bc(A))(\phi_{z_{\gamma}}(\zeta))\\
&= b_x(\zeta)^{-1} \lim\limits_{z_{\gamma} \to x} (\Bc(A))(\phi_{z_{\gamma}}(\zeta))
\end{align*}
for every $\zeta \in \Omega$. In particular, setting $\zeta = 0$ and using that $b_z(0) = 1$ for all $z \in \Omega$, we get
\begin{equation} \label{thm7_eq0}
(\Bc(A_xT_{b_x}^{-1}))(0) = \lim\limits_{z_{\gamma} \to x} (\Bc(A))(\phi_{z_{\gamma}}(0)) = \lim\limits_{z_{\gamma} \to x} (\Bc(A))({z_{\gamma}}) = (\Bc(A))(x).
\end{equation}
Now let $\zeta \in \overline{D(0,1)}$, i.e.~$\beta(0,\zeta) \leq 1$. Then
\begin{align*}
\lim\limits_{z_{\gamma} \to x} \abs{(\Bc(A))(\phi_{z_{\gamma}}(\zeta)) - (\Bc(A))(x)} &\leq \lim\limits_{z_{\gamma} \to x} \abs{(\Bc(A))(\phi_{z_{\gamma}}(\zeta)) - (\Bc(A))(\phi_{z_{\gamma}}(0))}\\
&\qquad + \lim\limits_{z_{\gamma} \to x} \abs{(\Bc(A))(\phi_{z_{\gamma}}(0)) - (\Bc(A))(x)}\\
&\leq \lim\limits_{z_{\gamma} \to x} \Osc_{z_{\gamma}}(\Bc(A))\\
&= 0
\end{align*}
because of Equation \eqref{thm7_eq0} and $\beta(\phi_{z_{\gamma}}(0),\phi_{z_{\gamma}}(\zeta)) = \beta(0,\zeta) \leq 1$. By repeating this argument, we see that this generalizes to all $\zeta \in \Omega$. Thus
\begin{equation} \label{thm7_eq1}
(\Bc(A_xT_{b_x}^{-1}))(\zeta) = b_x(\zeta)^{-1} (\Bc(A))(x)
\end{equation}
for all $\zeta \in \Omega$. Now consider the case $A = I$. If $A = I$, then obviously $A_x = I_x = I$ for all $x \in \beta\Omega \setminus \Omega$ and hence
\begin{equation} \label{thm7_eq2}
(\Bc(T_{b_x}^{-1}))(\zeta) = (\Bc(I_xT_{b_x}^{-1}))(\zeta) = b_x(\zeta)^{-1} (\Bc(I))(x) = b_x(\zeta)^{-1}
\end{equation}
for all $\zeta \in \Omega$. Combining \eqref{thm7_eq1} and \eqref{thm7_eq2}, we get
\[\Bc\left(\left(A_x - (\Bc(A))(x) \cdot I\right)T_{b_x}^{-1}\right) = 0\]
by linearity and $A_x = (\Bc(A))(x) \cdot I$ by the injectivity of $\Bc$ (see Section \ref{compactness}).

Conversely, assume that $A_x$ is a multiple of the identity for every $x \in \beta\Omega \setminus \Omega$, i.e.~$A_x = \lambda_x \cdot I$ for some $\lambda_x \in \C$. Choose a net $(z_{\gamma})$ in $\Omega$ that converges to $x \in \beta\Omega \setminus \Omega$. Using Equation \eqref{Berezin_transform_shifted_operator} again, we get
\[b_x(\zeta)^{-1} \lim\limits_{z_{\gamma} \to x} (\Bc(A))(\phi_{z_{\gamma}}(\zeta)) = \lim\limits_{z_{\gamma} \to x} (\Bc(A_{z_{\gamma}}T_{b_{z_{\gamma}}}^{-1}))(\zeta) = \lambda_x \cdot (\Bc(T_{b_x}^{-1}))(\zeta)\]
for all $\zeta \in \Omega$. Equation \eqref{thm7_eq2} thus implies
\begin{equation} \label{thm7_eq3}
\lim\limits_{z_{\gamma} \to x} (\Bc(A))(\phi_{z_{\gamma}}(\zeta)) = \lambda_x
\end{equation}
for every $\zeta \in \Omega$ and $x \in \beta\Omega \setminus \Omega$. In particular, setting $\zeta = 0$, $\lambda_x = (\Bc(A))(x)$. Now assume that $\Bc(A)$ is not contained in $\VO_{\partial}(\Omega)$. Then there are $\epsilon > 0$ and two sequences $(z_n)_{n \in \N}$, $(w_n)_{n \in \N}$ with $\beta(z_n,w_n) \leq 1$ such that $w_n \to \partial\Omega$ and
\begin{equation} \label{thm7_eq4}
\abs{(\Bc(A))(z_n) - (\Bc(A))(w_n)} > \epsilon
\end{equation}
for all $n \in \N$. Since $\beta(0,\phi_{z_n}(w_n)) = \beta(z_n,w_n) \leq 1$ for all $n \in \N$, we can assume without loss of generality that the sequence $(\phi_{z_n}(w_n))_{n \in \N}$ converges to some $\zeta \in \overline{D(0,1)}$. Moreover, we may take a subnet $(z_{\gamma})$ of $(z_n)_{n \in \N}$ that converges to some $x \in \beta\Omega \setminus \Omega$. The corresponding subnet of $(w_n)_{n \in \N}$ we denote by $(w_{\gamma})$. Consider
\begin{align*}
\abs{(\Bc(A))(w_{\gamma}) - (\Bc(A))(z_{\gamma})} &\leq \abs{(\Bc(A))(w_{\gamma}) - (\Bc(A))(\phi_{z_{\gamma}}(\zeta))} + \abs{(\Bc(A))(\phi_{z_{\gamma}}(\zeta)) - (\Bc(A))(\phi_{z_{\gamma}}(0))}.
\end{align*}
The second term on the right-hand side tends to $0$ by Equation \eqref{thm7_eq3}. For the first term we observe that $\beta(w_{\gamma},\phi_{z_{\gamma}}(\zeta)) = \beta(\phi_{z_{\gamma}}(w_{\gamma}),\zeta)$ tends to $0$ by construction and since $\Bc(A)$ is uniformly continuous, the first term tends to $0$ as well. But this is a contradiction to \eqref{thm7_eq4}. Therefore $\Bc(A)$ has to be contained in $\VO_{\partial}(\Omega)$.
\end{proof}

\begin{cor} \label{cor4}
Let $A \in \Lc(A^p_{\nu})$ be band-dominated with $\Bc(A) \in \VO_{\partial}(\Omega)$. Then
\[\spec_{\ess}(A) = \bigcup\limits_{x \in \beta\Omega \setminus \Omega} (\Bc(A))(x) = (\Bc(A))(\beta\Omega \setminus \Omega) = \bigcap\limits_{r > 0} \overline{(\Bc(A))(\Omega \setminus D(0,r))}.\]
\end{cor}

In what follows we will use the standard abbreviation $\tilde{f} := \Bc(T_f)$. Note that this definition is independent of $p$. A bounded function $f$ is then called of bounded mean oscillation, denoted as $f \in \VMO_{\partial}(\Omega) \cap L^{\infty}(\Omega)$, if
\[(\MO(f))(z) := \left(|f - \tilde{f}(z)|^2\right)^{\sim}(z) \to 0\]
as $z \to \partial\Omega$.

\begin{cor} \label{cor5}
Let $f \in \VMO_{\partial}(\Omega) \cap L^{\infty}(\Omega)$. Then
\[\spec_{\ess}(T_f) = \tilde{f}(\beta\Omega \setminus \Omega) = \bigcap\limits_{r > 0} \overline{\tilde{f}(\Omega \setminus D(0,r))}.\]
\end{cor}

\begin{proof}
By Corollary \ref{cor4} it suffices to show that $\tilde{f} = \Bc(T_f)$ is contained in $\VO_{\partial}(\Omega)$. For the unweighted case this was shown in \cite[Theorem F, Corollary 2]{BBCZ}. The same proof also applies to the weighted case and was (essentially) carried out in \cite[Proposition 4.8]{BaCo}.
\end{proof}

\bigskip

\noindent
Raffael Hagger\\
Institut f\"ur Analysis\\
Leibniz Universit\"at Hannover\\
Welfengarten 1\\
30167 Hannover\\
GERMANY\\
raffael.hagger@math.uni-hannover.de

\end{document}